\documentclass[opre,nonblindrev]{informs3} 

\DoubleSpacedXI 

\usepackage{endnotes}
\let\footnote=\endnote
\let\enotesize=\normalsize
\def\notesname{Endnotes}%
\def\makeenmark{$^{\theenmark}$}
\def\enoteformat{\rightskip0pt\leftskip0pt\parindent=1.75em
  \leavevmode\llap{\theenmark.\enskip}}

\usepackage{natbib}
 \bibpunct[, ]{(}{)}{,}{a}{}{,}%
 \def\bibfont{\small}%
\TheoremsNumberedThrough     
\ECRepeatTheorems

\EquationsNumberedThrough    

\usepackage{enumerate}
\usepackage{amsmath}
\usepackage{mathtools}
\usepackage{verbatim}
\usepackage{algorithmic}
\usepackage{algorithm}
\usepackage{booktabs}
\usepackage{multirow}
\usepackage{setspace}
\usepackage{algorithm}
\usepackage{algorithmic}

\newcommand{\Black}[1]{\textbf{\textcolor{black}{#1}}}
\def\argmax{\mathop{\arg\max}}

\begin{document}


\RUNAUTHOR{Yang et al.}

\RUNTITLE{Solving HOP Problems Via MINLP}

\TITLE{Solving Heated Oil Pipeline Problems Via Mixed Integer Nonlinear Programming Approach}

\ARTICLEAUTHORS{%
\AUTHOR{Muming Yang}
\AFF{LSEC, ICMSEC, Academy of Mathematics and Systems Science, Chinese Academy of Sciences, Beijing, 100190, \\
Mathematical Sciences, University of Chinese Academy of Sciences, Beijing, 100049, \EMAIL{ymm@lsec.cc.ac.cn}} 
\AUTHOR{Yakui Huang}
\AFF{Institute of Mathematics, Hebei University of Technology, Tianjin, 300401, \EMAIL{huangyakui2006@gmail.com}}
\AUTHOR{Yu-Hong Dai}
\AFF{LSEC, ICMSEC, Academy of Mathematics and Systems Science, Chinese Academy of Sciences, Beijing, 100190, \\
Mathematical Sciences, University of Chinese Academy of Sciences, Beijing, 100049, \EMAIL{dyh@lsec.cc.ac.cn}}
\AUTHOR{Bo Li}
\AFF{CNPC Key Laboratory of Oil \& Gas Storage and Transportation, PetroChina Pipeline R \& D Center, Langfang, 065000, \EMAIL{libocolby@yeah.net}}
} 

\ABSTRACT{%
It is a crucial problem how to heat oil and save running cost for crude oil transport. This paper strictly formulates such a heated oil pipeline problem as a mixed integer nonlinear programming model. Nonconvex and convex continuous relaxations of the model are proposed, which are proved to be equivalent under some suitable conditions. Meanwhile, we provide a preprocessing procedure to guarantee these conditions. Therefore we are able to design a branch-and-bound algorithm for solving the mixed integer nonlinear programming model to global optimality. To make the branch-and-bound algorithm more efficient, an outer approximation method is proposed as well as the technique of warm start is used. The numerical experiments with a real heated oil pipeline problem show that our algorithm achieves a better scheme and can save 6.83\% running cost compared with the practical scheme.
}%



\maketitle

%


\section{Introduction}

Crude oil, as the raw material of petroleum products, is critical to the industry and daily life. Before refining, crude oil needs to be transported from oil fields to refineries. According to incomplete statistics, 51\% of the oil around the world is transported via pipelines.
During the transport, it is often necessary to pressurize the oil to keep it run through the whole pipeline safely.
Meanwhile, the oil requires to be heated up in case of congelation and high viscosity. For example, the condensation point of the oil produced in Daqing oilfield of China reaches 32$^\circ$C. In this case, there are not only pumps but also heating furnaces equipped in heated oil pipeline stations (see Fig. \ref{pumpandfurnace}). The energy consumed by heating furnaces is approximately equivalent to 1\% of the oil transported in the pipeline. Therefore even for the long distance heated oil pipeline (HOP), it is crucial to optimize the operation scheme so as to save the transport cost.

\begin{figure}[H]
	\FIGURE
	{\includegraphics[width=0.6\textwidth]{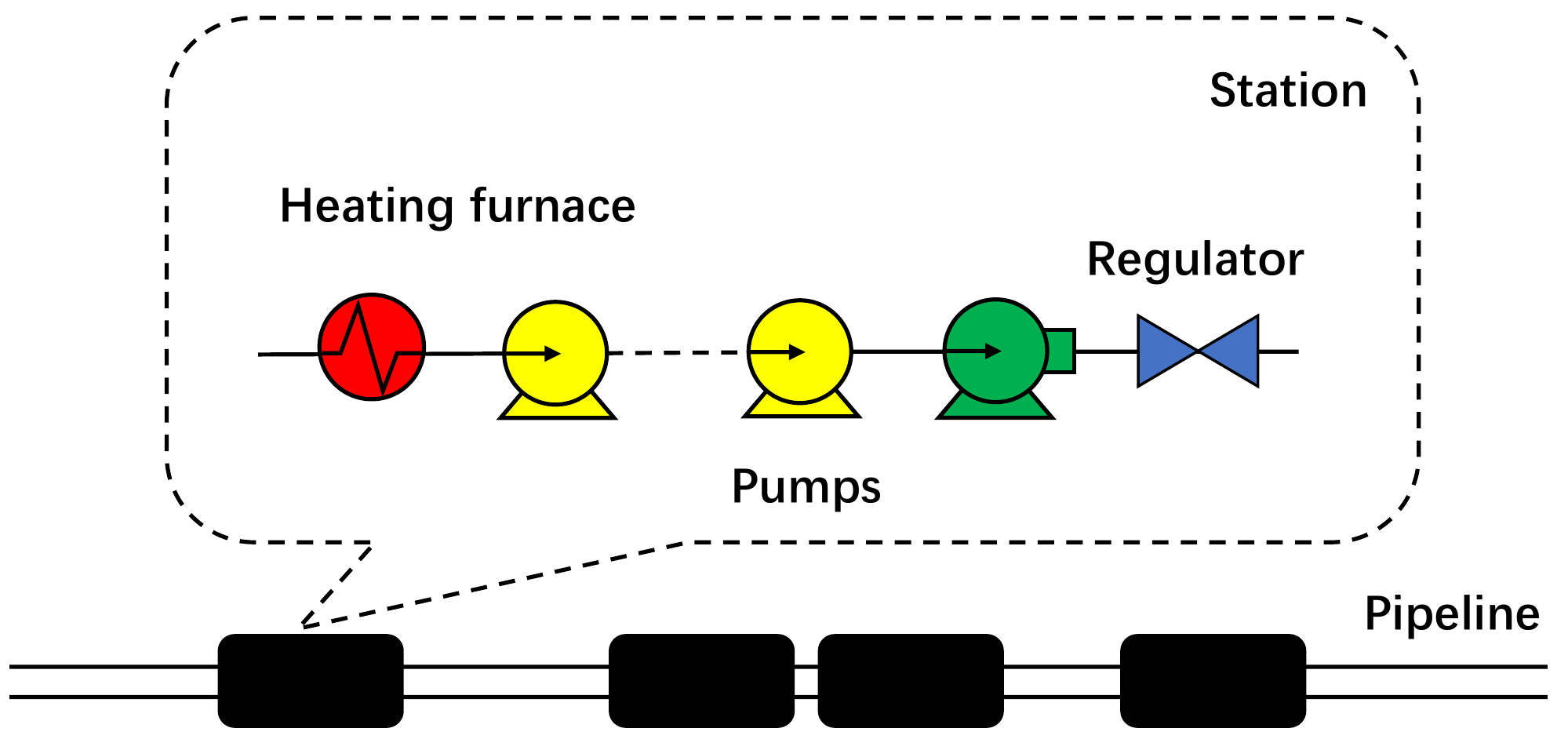}}
	{Pipeline, Station, Heating Furnace, Regulator and Pumps, 
    Including Constant Speed Pumps (Yellow) and Shifted Speed Pumps (Green and with A Speed Controller).\label{pumpandfurnace}}
	{}
\end{figure}

An operation scheme of the HOP mainly consists of pump combination and furnace operation in each station. To meet the safe transport requirements, such a scheme has to satisfy some constraints, such as inlet and outlet pressure and temperature bounds and transition points constraints. In practice, feasible schemes always exist as long as sufficient and proper pressure and heat is provided for the oil flow. However, the costs of different operation schemes may vary tremendously. Higher temperature of the oil consumes more heating cost which, at the same time, usually allows pumps to produce lower pressure for the oil to arrive the destination safely. Thus the optimal total cost is closely related to the combinatorial relation of the two kinds of facilities. The HOP problem is just to figure out a combination of pumps and heating furnaces which satisfies the feasibility requirement and, meanwhile, minimizes the total cost of the power consumption, including both electricity and fuel consumed by the two kinds of facilities.


There are quite a few researches on the oil pipeline problem in the past forty years. To our knowledge, \cite{gopal1980optimizing} gave the first model for the isothermal oil pipeline problem and analyzed the optimal selection of pump combination and discharge pressure via some dynamic and integer programming techniques. In 2001, \citeauthor{jokic2001optimization} proposed another isothermal oil pipeline model via nonlinear programming and tried to optimize the diameters of pipeline for saving running cost. For the HOP problem, \cite{wu1989two,wu1992introduction} designed a two-level hierarchical model based on decomposition and built a software called \verb|HOPOPT|.  \cite{meng2002optimization} implemented a nonlinear programming and station-by-station method to optimize HOP problems. Other studies based on meta-heuristic approaches can be found in \cite{zhou2015dynamic} and \cite{liu2015research}. See \cite{Wang2012A} for a survey on oil/gas pipeline optimization.

Notice that in pumping and heating stations, there are constant speed pumps (CSPs) and shifted speed pumps (SSPs), which increase fixed amounts of pressure and continually variable pressure, respectively. By modeling the on-off status of CSPs and SSPs as integer variables and modeling the head provided by SSPs and the temperature rise of furnaces as continuous variables, this paper shall strictly formulate the HOP problem as a mixed integer nonlinear programming (MINLP) model and consider the algorithmic designing within the branch-and-bound framework \citep{grossmann1997mixed} to find the global optimum. For solving convex MINLP, \cite{quesada1992lp} developed an LP/NLP based branch-and-bound algorithm and \cite{fletcher1994solving} proposed an outer approximation based method. See \verb|Bonmin| (\url{https://www.coin-or.org/Bonmin/}) and \verb|FilMINT| \citep{abhishek2010filmint} for some more convex MINLP solvers.
However, there is an intrinsic difficulty in solving our MINLP model due to the nonconvexity of the head loss constraints. The hydraulic friction of oil flows, as an important component of head loss constraints, is usually highly complicated to evaluate. Based on the Darcy-Weisbach formula \citep{darcy1857recherches}, the hydraulic friction $\mathrm{HF}$ can be calculated by
\begin{equation}\label{darcy}
\mathrm{HF}\left(T, Q, D, L\right) = \beta(T) \frac{Q^{2-m(T)}}{D^{5-m(T)}}\left[\nu\left(T\right)\right]^{m(T)}L,
\end{equation}
where $T$ is the temperature of oil, $Q$ is the volume flow of oil, $D$ is the inner diameter of pipe, $L$ is the length of pipe, $\nu(\cdot)$ represents the kinematic viscosity of oil, and $\beta(\cdot)$ and $m(\cdot)$ are piecewise constant functions. There are some techniques and softwares for handling general nonconvex MINLP, such as piecewise linear approximation, spatial branch-and-bound \citep{horst2013global} and \verb|Couenne| (\url{https://www.coin-or.org/Couenne/}). One can refer to the surveys by \cite{Burer2012Non} and \cite{belotti_kirches_leyffer_linderoth_luedtke_mahajan_2013} for more details. In general, it is hard to achieve the global optimum since the region is nonconvex even if all integer variables are relaxed to continuous ones.

The contribution of this work lies in four folds. Firstly, we establish an MINLP model for HOP problems. Secondly, after analyzing a nonconvex relaxation and a convex relaxation of the original MINLP model, we prove the equivalence of the two relaxations under some conditions and meanwhile, a preprocessing procedure is proposed to guarantee these conditions. This enables us to design a branch-and-bound algorithm to obtain the global optimum in a finite number of iterations.
Thirdly, to improve the efficiency, an outer approximation method is implemented for solving the subproblems as well as some warm start strategy is utilized. Finally, numerical experiments
with a real heated oil pipeline problem are conducted, which show that our algorithm achieves a better scheme and can save 6.83\% running cost compared with the practical scheme.

The rest of this paper is organized as follows. In section \ref{sec_MINLP}, we formulate the HOP problem by an MINLP model after giving some notations. Section \ref{sec_Relaxation_problems} addresses the nonconvex and convex relaxations of the MINLP model and establishes their equivalence under some conditions. Section \ref{sec_Algorithms} provides a preprocessing procedure which guarantees these conditions. Furthermore, combing an outer approximation method and a warm start strategy for solving subproblems, we provide the branch-and-bound algorithm for solving the HOP problem in this section. Numerical experiments are presented in section \ref{sec_Numerical_results} and some concluding remarks are given in section \ref{sec_Conclusions}.


\section{The MINLP Model with A Complexity Analysis}\label{sec_MINLP}

In this section, we first give a brief description of the HOP problem and some basic assumptions and notations. Then we present an MINLP model for the HOP problem.

\subsection{Problem Descriptions, Assumptions and Notations}\label{sec_Notations}
During the transport of heated oil, heat dissipation and friction cause the drop of the oil temperature. Meanwhile, hydraulic friction and elevation difference result in the drop of the oil pressure. To meet the requirements including safe inlet and outlet temperature and pressure of each station and safe pressure at each transition point, it is crucial to maintain proper temperature and pressure of the oil flow via pumps and furnaces in each station. Hence, main decisions to be made in HOP is whether each pump should be powered up, which defines discrete variables, and what temperature should the oil be heated to, which defines continuous variables. The target of the HOP problem is to figure out the most economical scheme among all the feasible decisions. To focus on the main character of HOP problems and illustrate our idea conveniently, necessary assumptions are made for the rest of this paper.
\begin{assumption}\label{assumption_fund}
Suppose the oil flow is in steady state and the influence of friction heat to the temperature of oil is a constant value. Moreover, the following assumptions are made on pumps and furnaces in the pipeline.
\begin{enumerate}[(i)]
	\item The powered up CSPs (if exist) in the same station have the same head value and efficiency;
	\item The SSPs (if exist) in the same station have the same lower, upper head bounds and efficiency;
	\item The heating furnaces in the same station have the same efficiency;
	\item All furnaces consume natural gas;
	\item The thermal load of furnaces is unlimited;
	\item The inlet head and temperature in the first station are given;
	\item There is no pump or furnace in the last station.
\end{enumerate}
\end{assumption}
If not specified, Assumption \ref{assumption_fund} is used in our analysis throughout this paper. It is common that the steady state of oil flow is assumed (see \cite{wu1989two}, \cite{Li2011Operation} and \cite{liu2015research}), in which the volume flow of oil is invariant. For convenience, we specify some settings on pumps and furnaces in the above assumption. Some of them are based on practical situations, such as the assumptions (i), (ii), (iii), (iv), (vi) and (vii). The assumption (vi) shows that the heating furnaces will not consume the oil transported in the pipeline. For the assumption (v), in most cases, the thermal load of furnaces can be limited by the upper bound of outlet temperature in each station. The discussion on friction heat in Assumption \ref{assumption_fund} will take place in the last section of this paper. The head and efficiency assumptions on pumps and furnaces enable us easily to evaluate the operation cost in HOP problems and simplify the solution procedure partly, which is not the point.
For convenience, we describe the constants and variables used for the model description in Tables \ref{table_cnotations} and \ref{table_vnotations}, respectively.
\begin{table}[!h]
	\TABLE
	{Constants Used for Modelling. \label{table_cnotations}}
	{\begin{tabular}{rl}
		\toprule
		$N^S$ & number of station(s) \\
		$N^{CP}_j$ & number of CSP(s) in station $j$ \\
		$N^{SP}_j$ & number of SSP(s) in station $j$ \\
		$N^P_j$ & number of pipe segment(s) between stations $j$ and $j+1$ \\
		$\rho$ & density of oil [kg/m$^3$] \\
		$\pi$ & circumference ratio \\
		$c$ & specific heat of oil [J/(kg$\cdot ^\circ$C)] \\
		$g$ & acceleration of gravity [m/s$^2$] \\
		$C_p$ & unit-price of electricity consumed by pump(s) [yuan/(W$\cdot$s)] \\
		$C_f$ & unit-price of fuel (gas) consumed by heating furnace(s) [yuan/m$^3$] \\
		$V_c$ & heat value of fuel (gas) consumed by heating furnace(s) [J/m$^3$] \\
		$\xi^{CP}_j$ & efficiency of CSP(s) in station $j$ \\
		$\xi^{SP}_j$ & efficiency of SSP(s) in station $j$ \\
		$\eta_j$ & efficiency of heating furnace(s) in station $j$ \\
		$K_{jr}$ & heat transfer coefficient at the $r$-th pipe segment between stations $j$ and $j+1$ [W/(m$^2\cdot ^\circ$C)] \\
		$Q_{jr}$ & volume flow at the $r$-th pipe part between stations $j$ and $j+1$ [m$^3$/h] \\
		$L_{jr}$ & pipe length of the $r$-th pipe segment between stations $j$ and $j+1$ [m] \\
		$D_{jr}$ & inner diameter of the $r$-th pipe segment between stations $j$ and $j+1$ [m] \\
		$d_{jr}$ & outer diameter of the $r$-th pipe segment between stations $j$ and $j+1$ [m] \\
		$\Delta Z_{jr}$ & elevation difference of the $r$-th pipe segment between stations $j$ and $j+1$ [m] \\
		$H^{CP}_j$ & head of CSP(s) in station $j$ [m] \\
		$\underline{H}^{SP}_j (\overline{H}^{SP}_j)$ & minimal (maximal) head of SSP(s) in station $j$ [m] \\
		$T_g^{P_{jr}}$ & ground temperature at the $r$-th pipe segment between stations $j$ and $j+1$ [$^\circ$C] \\
		$T_f^{P_{jr}}$ & temperature changes caused by friction heat at the $r$-th pipe segment between stations $j$ and $j+1$ [$^\circ$C] \\
		$\underline{x}_j (\overline{x}_j)$ & minimal (maximal) number of the powered up CSP(s) in station $j$ \\
		$\underline{y}_j (\overline{y}_j)$ & minimal (maximal) number of the powered up SSP(s) in station $j$ \\
		$\underline{H}_{in}^{S_j} (\underline{H}_{out}^{S_j})$ & lower bound of the inlet (outlet) head in station $j$ [m] \\
		$\overline{H}_{in}^{S_j} (\overline{H}_{out}^{S_j})$ & upper bound of the inlet (outlet) head in station $j$ [m] \\
		$\underline{T}_{in}^{S_j} (\underline{T}_{out}^{S_j})$ & lower bound of the inlet (outlet) temperature in station $j$ [$^\circ$C] \\
		$\overline{T}_{in}^{S_j}(\overline{T}_{out}^{S_j})$ & upper bound of the inlet (outlet) temperature in station $j$ [$^\circ$C] \\
		$\underline{H}_{out}^{P_{jr}}(\overline{H}_{out}^{P_{jr}})$ & lower (upper) bound of the head at the end of the $r$-th pipe segment between stations $j$ and $j+1$ [m] \\
		\bottomrule
	\end{tabular}}
	{}
\end{table}

\begin{table}[!h]
	\TABLE
	{Variables Used for Modelling. \label{table_vnotations}}
	{\begin{tabular}{rl}
		\toprule
		$x_j$ & number of powered up CSP(s) in station $j$ \\
		$y_j$ & number of powered up SSP(s) in station $j$ \\
		$\Delta H^{SP}_j$ & head of shifted speed pump in station $j$ [m] \\
		$\Delta T_j$ & temperature rise in station $j$ [$^\circ$C] \\
		$H_{in}^{S_j}(H_{out}^{S_j})$ & inlet (outlet) head in station $j$ [m] \\
		$T_{in}^{S_j}(T_{out}^{S_j})$ & inlet (outlet) temperature in station $j$ [$^\circ$C] \\
		$H_{out}^{P_{jr}}$ & head at the end of the $r$-th pipe segment between stations $j$ and $j+1$ [m] \\
		$T_{out}^{P_{jr}}$ & temperature at the end of the $r$-th pipe segment between stations $j$ and $j+1$ [$^\circ$C] \\
		$T_{ave}^{P_{jr}}$ & average temperature at the $r$-th pipe segment between stations $j$ and $j+1$ [$^\circ$C] \\
		$F_{jr}$ & hydraulic friction at the $r$-th pipe part between stations $j$ and $j+1$ [m] \\
		\bottomrule
	\end{tabular}}
	{}
\end{table}

Fig. \ref{fig_notations} shows these notations in detail with a pipeline. Note that based on practice, all constants and parameters related to price, efficiency, volume flow, pipe length, pipe diameter, oil head, oil temperature and friction are nonnegative.

\begin{figure}[!h]
\FIGURE
{\includegraphics[width=0.9\textwidth]{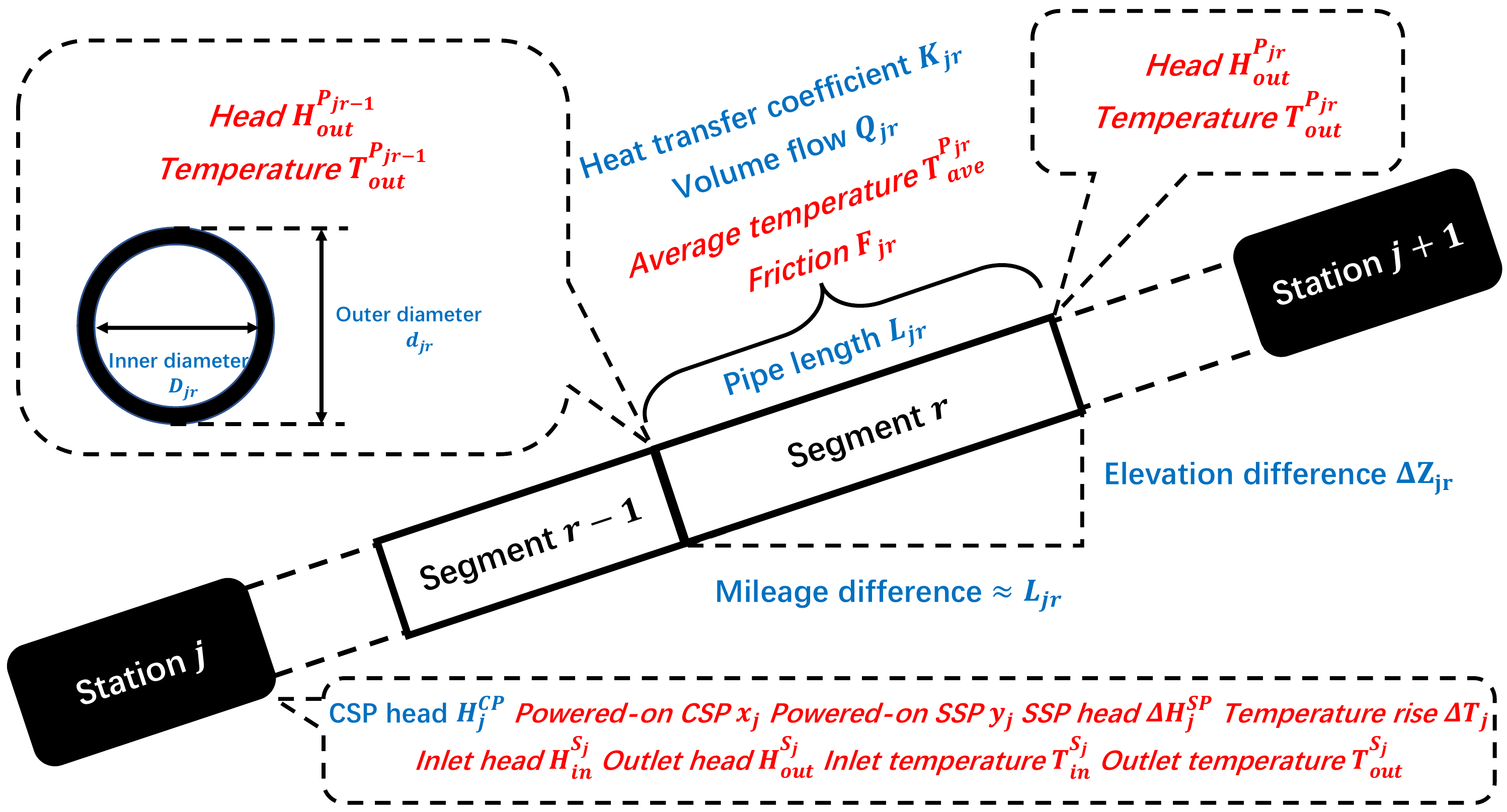}}
{Constants (Blue and Nonitalic) and Variables (Red and Italic) in a Pipeline. \label{fig_notations}}
{}
\end{figure}

\subsection{The MINLP Model with A Complexity Analysis}\label{sec_MINLP_model}

%

Suppose there are $N^S$ stations in the pipeline in our model. In the HOP problem, tracking the head loss and the temperature drop is important and complicated. For the purpose of calculating the variation of head and temperature in the pipeline accurately and limiting heads at transition points properly, the pipe between stations $j$ and $j+1$ is divided into $N^P_j$ segments, $j=1,...,N^S-1$. Then the heads at the two sides of each pipe segment have the following relation,
\begin{align}
\label{constr_hchangepart}
H_{out}^{P_{jr}} &= H_{out}^{P_{j,r-1}} - F_{jr} - \Delta Z_{jr},~ \ \ j=1,...,N^S-1,~ \ \ r=1,...,N^{P}_j,
\end{align}
where $H_{out}^{P_{j,r-1}}$ and $H_{out}^{P_{jr}}$ represent the heads at the start and the end of the $r$-th pipe segment, respectively.  We can see that the head loss in the pipeline consists of two components: the friction $F_{jr}$ and the elevation difference $\Delta Z_{jr}$. Rather than formulating $F_{jr}$ in a specific way shown in the formula \eqref{darcy}, a general nonlinear function $f$ is used in this model as follows
\begin{align}
	\label{constr_nonlinear}
	F_{jr} &= f\left(T_{ave}^{P_{jr}}, Q_{jr}, D_{jr}\right)L_{jr},~ \ \ j=1,...,N^S-1,~ \ \ r=1,...,N^{P}_j.
\end{align}
Here $T_{ave}^{P_{jr}}$ is the average temperature in the $r$-th pipe segment. To figure out $T_{ave}^{P_{jr}}$, we first evaluate $T_{out}^{P_{jr}}$, the temperature at the end of the pipe segment, via the axial temperature drop formula as follows
\begin{equation}
\label{constr_outTpart}
T_{out}^{P_{jr}} = T_{g}^{P_{jr}} + T_f^{P_{jr}} + \left[T_{out}^{P_{j,r-1}} - \left(T_{g}^{P_{jr}} + T_f^{P_{jr}}\right)\right] e^{-\alpha_{jr}L_{jr}},~ \ \ j=1,...,N^S-1,~ \ \ r=1,...,N^{P}_j,
\end{equation}
where $T_{g}^{P_{jr}}$ is the ground temperature, $T_f^{P_{jr}}$ is the temperature change caused by friction heat, $\alpha_{jr}$ is a parameter defined as
\begin{equation*}
\alpha_{jr} = \frac{K_{jr}\pi d_{jr}}{\rho Q_{jr} c},~ \ \ j=1,...,N^S-1,~ \ \ r=1,...,N^{P}_j.
\end{equation*}
In the definition of $\alpha_{jr}$, $K_{jr}$ refers to the heat transfer coefficient, $\rho$ is the density of oil, $c$ denotes the specific heat of oil. The average temperature in the pipe $T_{ave}^{P_{jr}}$ is usually obtained based on production experience. That is,
\begin{equation}
\label{constr_aveTpart}
T_{ave}^{P_{jr}} = \frac{1}{3}T_{out}^{P_{j,r-1}} + \frac{2}{3}T_{out}^{P_{jr}},~ \ \ j=1,...,N^S-1,~ \ \ r=1,...,N^{P}_j.
\end{equation}

To keep the head and pressure on an appropriate level, supplements of pressure and heat are essential. A CSP can only provide a fixed amount of head, while a continually variable pressure can be produced by an SSP. For the pumps in the $j$-th station ($j=1,...,N^S-1$), integer variables $x_j$ and $y_j$ are introduced to represent the number of powered up CSPs and SSPs, respectively. A continuous variable $\Delta H^{SP}_j$ is used to represent the total head produced by SSPs. The upper bound of the outlet head of each station is decided by the inlet head and the total head change in the station. That is,

\begin{align}
\label{constr_inoutH} H_{in}^{S_j} + x_jH^{CP}_j + \Delta H^{SP}_j &\geq H_{out}^{S_j},~ \ \ j = 1,...,N^S-1.
\end{align}
Note that there is a pressure regulator in each station. In case of dangerous outlet pressure, the regulator will be switched on to lower the head of the oil before it flows out the station. Therefore, the constraints in (\ref{constr_inoutH}) are all inequalities.
As a matter of fact, these inequalities reflect the special structure of HOP problems. The head of a single SSP is not only upper but also lower bounded in practice. This limitation can be formulated as
\begin{equation}
	y_j\underline{H}_j^{SP} \leq \Delta H^{SP}_j \leq y_j\overline{H}^{SP}_j,~ \ \ j=1,...,N^S-1.
\end{equation}
For heating furnaces, we denote a continuous variable $\Delta T_j$ as the temperature rise in the $j$-th station ($j=1,...,N^S-1$). Then the outlet temperature can be obtained by
\begin{equation}
\label{constr_inoutT} T_{in}^{S_j} + \Delta T_j = T_{out}^{S_j},~ \ \ j = 1,...,N^S-1.
\end{equation}
To establish the connection between stations and pipes, $T_{out}^{P_{j0}}$ and $H_{out}^{P_{j0}}$ are defined as the outlet temperature and head of the $j$-th station, respectively. Meanwhile, when $r=N^P_j$, let $T_{out}^{P_{jr}}$ and $H_{out}^{P_{jr}}$ be the inlet temperature and head in station $j+1$, respectively. Then we have
\begin{align}
\label{constr_connectT}
T_{out}^{P_{j0}} &= T_{out}^{S_j},~ \ \ T_{out}^{P_{jN^P_j}} = T_{in}^{S_{j+1}},~  \ \ j=1,...,N^S-1,\\
\label{constr_connectH}
H_{out}^{P_{j0}} &= H_{out}^{S_j},~ \ \ H_{out}^{P_{jN^P_j}} = H_{in}^{S_{j+1}},~ \ \ j=1,...,N^S-1.
\end{align}

Moreover, there are also some bound constraints in this model. In each station, the number of powered up CSPs and SSPs are bounded (see \eqref{constr_xbound}, \eqref{constr_ybound}). Only nonnegative temperature rise is allowed (see \eqref{constr_temprisebound}). The inlet and outlet heads and temperatures are also upper and lower bounded by constraints (see \eqref{constr_boundHin}-\eqref{constr_boundTout}). Furthermore, the head restriction is implemented in each pipe segment (see \eqref{constr_transitionHead}). These bound constraints cover the transition point requirement. We summarize all the bound constraints as follows.
\begin{alignat}{2}
\label{constr_xbound}
x_j &\in \{\underline{x}_j, \underline{x}_j+1,..., \overline{x}_j\},~ \ \ &&j=1,...,N^S-1, \\
\label{constr_ybound}
y_j &\in \{\underline{y}_j, \underline{y}_j+1,..., \overline{y}_j\},~ \ \ &&j=1,...,N^S-1, \\
\label{constr_temprisebound}
\Delta T_j &\geq 0,~ \ \ &&j=1,...,N^S-1, \\
\label{constr_boundHin}
\underline{H}_{in}^{S_j} &\leq H_{in}^{S_j} \leq \overline{H}_{in}^{S_j},~ \ \ &&j=1,...,N^S, \\
\label{constr_boundHout}
\underline{H}_{out}^{S_j} &\leq H_{out}^{S_j} \leq \overline{H}_{out}^{S_j},~ \ \ &&j=1,...,N^S-1, \\
\label{constr_boundTin}
\underline{T}_{in}^{S_j} &\leq T_{in}^{S_j} \leq \overline{T}_{in}^{S_j},~ \ \ &&j=1,...,N^S, \\
\label{constr_boundTout}
\underline{T}_{out}^{S_j} &\leq
T_{out}^{S_j} \leq \overline{T}_{out}^{S_j},~ \ \ &&j=1,...,N^S-1, \\
\label{constr_transitionHead} \underline{H}_{out}^{P_{jr}} &\leq H_{out}^{P_{jr}} \leq \overline{H}_{out}^{P_{jr}},~ \ \ &&j=1,...,N^S-1,~ \ \ r=1,...,N^{P}_j-1.
\end{alignat}
Note that based on (vi) in Assumption \ref{assumption_fund}, it is true that
\begin{equation*}
	\underline{H}_{in}^{S_1} = \overline{H}_{in}^{S_1},~ \ \ \underline{T}_{in}^{S_1} = \overline{T}_{in}^{S_1}.
\end{equation*}
Finally, the goal of the HOP problem is to minimize the total cost of pumps and heating furnaces, namely, electricity cost per hour coming from powered up pumps and fuel or gas cost per hour coming from switched on furnaces. More exactly, the total cost function is
\begin{equation}\nonumber
C(x, \Delta H^{SP}, \Delta T) = \sum_{j=1}^{N^S-1}\left[
C_p\rho Q_{j0} g\left(\frac{x_j H^{CP}_j}{\xi^{CP}_{j}}+\frac{\Delta H^{SP}_j}{\xi^{SP}_{j}}\right) + C_fc\rho Q_{j0} \frac{\Delta T_j}{\eta_jV_c}\right],
\end{equation}
where $C_p$ and $C_f$ are the unit prices of electricity and fuel, respectively, $\xi^{CP}$ and $\xi^{SP}$ are the efficiencies of CSPs and SSPs, respectively, $\eta$ is the efficiency of furnaces, $V_c$ is the heating value of the fuel.

Define $z \coloneqq \left(x, y\right)$ to be the vector representing the numbers of powered up CSPs and SSPs.
Meanwhile, let $\Psi$ denote the scheme accumulating all necessary quantities in HOP problems
\begin{equation*}
	\Psi \coloneqq \left(z, \Delta H^{SP}, \Delta T, H_{in}^S, H_{out}^S, T_{in}^S, T_{out}^S, H_{out}^P, T_{out}^P, T_{ave}^P, F\right).
\end{equation*}
Then in a simple way, we can formulate the HOP problem into the following MINLP model
\begin{equation*}
\text{(HOP($\underline{z}$, $\overline{z}$))}~~ \ \
	\begin{aligned}
	\min_{\Psi} \quad & C(x, \Delta H^{SP}, \Delta T) \\
	\text{s.t.} ~~~  & (\ref{constr_hchangepart}) - (\ref{constr_transitionHead}),
	\end{aligned}
\end{equation*}
where
	$\underline{z} = (\underline{x}, \underline{y})$ and $\overline{z} = \left(\overline{x}, \overline{y}\right).$
In the original HOP problem, we have that
\begin{equation*}
	\underline{z} = \left(0,0\right),~ \ \  \overline{z}=\left(N^{CP}, N^{SP}\right).
\end{equation*}
If the bounds of $x$ and $y$ are not specified, we just denote the above model by (HOP).

\begin{remark}\label{remark_sHOP}
Each choice of $s \coloneqq (z, \Delta H^{SP}, \Delta T, H_{out}^{S})$ defines a unique scheme $\Psi$ of (HOP) satisfying
all equality constraints. Specifically, $s$ defines a unique feasible scheme $\Psi$ of (HOP) if the scheme $\Psi$ satisfies all inequality constraints.
\end{remark}
Based on Remark \ref{remark_sHOP}, $s$ is regarded as a feasible solution of (HOP) if $s$ defines a feasible scheme $\Psi$ of (HOP). In general, MINLP problems are $\mathcal{NP}$-hard. The following proposition verifies the $\mathcal{NP}$-hardness of the problem (HOP) via the cutting stock problem
 (all proofs can be seen in the Appendix).
\begin{proposition}\label{prop_np}
	The problem (HOP) is $\mathcal{NP}$-hard.
\end{proposition}

Although the objective function and the remaining constraints are linear, the constraints \eqref{constr_nonlinear} are usually nonconvex and hence (HOP) is a nonconvex MINLP model. For a general nonconvex MINLP problem, it is difficult to find even the local optimum of its continuous relaxation problem. On the other hand, it is highly expected in practice if we could obtain the global optimum of (HOP). In the next sections, we
shall show that it is possible to seek the global optimum of (HOP) via a branch-and-bound method after some careful analysis.

\section{Nonconvex and Convex Relaxations and Their Equivalence}\label{sec_Relaxation_problems}

To treat the difficulties introduced by integer variables and nonconvex constraints, we propose both a nonconvex nonlinear relaxation and a convex nonlinear relaxation of the problem (HOP). The properties of the two relaxations will be discussed. Specifically, we prove that the two relaxations are equivalent under some conditions.

\subsection{A Nonconvex Nonlinear Relaxation of (HOP)}

In general, the continuous problem is more trackable than the discrete problem. By relaxing integer variables $x$ and $y$ in (HOP) to continuous ones, namely,
\begin{align}
	\label{constr_relaxxbound}
	x_j \in \left[\underline{x}_j, \overline{x}_j\right],~ \ \ j=1,...,N^S-1, \\
	\label{constr_relaxybound}
	y_j \in \left[\underline{y}_j, \overline{y}_j\right],~ \ \ j=1,...,N^S-1,
\end{align}
we get the following nonconvex relaxation of (HOP),
	\begin{equation*}
	\text{(HOPnr1($\underline{z},\overline{z}$))}~~ \ \
	\begin{aligned}
	\min_{\Psi} \quad & C(x, \Delta H^{SP}, \Delta T) \\
	\text{s.t.} ~~~  & (2)-(10), (13)-(20).
	\end{aligned}
	\end{equation*}
This is a continuous nonlinear programming (NLP) problem.

\begin{remark}\label{remark_s1}
	Each choice of $s \coloneqq (z, \Delta H^{SP}, \Delta T, H_{out}^{S})$ defines a unique scheme $\Psi$ of (HOPnr1) satisfying all equality constraints. Specifically, $s$ defines a unique feasible scheme $\Psi$ of (HOPnr1) if $\Psi$ satisfies all inequality constraints.
\end{remark}
Similar to (HOP), we regard $s$ as the solution vector and denote (HOPnr1) as the above model in short. Comparing with the original problem (HOP), (HOPnr1) has no integer variables any more and hence various methods for solving NLP can be applied \citep{bertsekas1997nonlinear}. For example, some kind of interior point method (IPM) can be implemented in solving (HOPnr1) directly if it is a smooth problem. However, it is difficult to achieve the global optimum of (HOPnr1) due to the nonconvexity of the constraints \eqref{constr_nonlinear}, which is a bad news for obtaining lower bounds of (HOP).

Fortunately, although we can not get the lower bound of (HOP) by solving (HOPnr1), each feasible solution of the relaxation will produce a feasible solution of the original problem (HOP), as shown in the following proposition.
\begin{proposition}\label{prop_upperbound}
	(HOP($\underline{z}$, $\overline{z}$)) is feasible if and only if (HOPnr1($\underline{z}$, $\overline{z}$)) is feasible. Particularly, assuming that $\check{s} \coloneqq (\check{z}, \Delta\check{H}^{SP}, \Delta\check{T}, \check{H}_{out}^S)$ is feasible to (HOPnr1($\underline{z}$, $\overline{z}$)), there exists a vector $\hat{s} \coloneqq (\hat{z}, \Delta \hat{H}^{SP}, \Delta \hat{T}, \hat{H}_{out}^S),$ where
	\begin{equation*}
		\begin{aligned}
		\hat{x} &= \lceil \check{x} \rceil,~ \ \
		\hat{y} = \lceil \check{y} \rceil,~ \ \
		\Delta \hat{T} = \Delta \check{T},~ \ \
		\hat{H}^S_{out} = \check{H}^S_{out}, \\
		\Delta \hat{H}_j^{SP} &= \left\{
			\begin{aligned}
				& \Delta \check{H}_j^{SP}, && ~\text{if\ }~ \hat{y}_j\underline{H}^{SP}_j \leq \Delta \check{H}_j^{SP} \leq \hat{y}_j\overline{H}^{SP}_j; \\
				& \hat{y}_j\underline{H}^{SP}_j, && ~\text{otherwise},	
			\end{aligned}
			\right.~ \ \ j=1,\dots,N^S-1,
		\end{aligned}
	\end{equation*}
	such that $\hat{s}$ is feasible to (HOP($\underline{z}$, $\overline{z}$)).
\end{proposition}

Proposition \ref{prop_upperbound} tells us, as long as we have a feasible solution of (HOPnr1), we can obtain an upper bound for (HOP) with little computational efforts. Therefore, the nodes in the branch-and-bound tree can be pruned with upper bounds obtained at the early stage. Such property is likely to reduce the amount of calculations.

Nevertheless, we have to face the fact that the global optimum of (HOPnr1) is hardly achievable due to nonconvexity.
In the next subsection, we derive another relaxation by relaxing the nonlinear nonconvex constraints.

\subsection{A Convex Nonlinear Relaxation of (HOP)}

To proceed, we make the following assumption on the function $f$ in (HOP).
\begin{assumption}\label{assup_fconvex}
	Given $Q_{jr} > 0$ and $D_{jr} > 0$, the function $f$ is convex and monotonically decreasing about $T_{ave}^{P_{jr}} > 0$.
\end{assumption}

Usually, the kinematic viscosity $\nu$ is convex and monotonically decreasing about the oil temperature. In addition, the oil flow is usually hydraulic smooth in practice. Recalling the friction formula \eqref{darcy}, it means that in the hydraulic smooth case, the parameters $\beta$ and $m$ are in certain constant pieces ($\beta \equiv 0.0246$, $m \equiv 0.25$) and therefore the friction $F$ has the same convexity and monotonicity with viscosity $\nu$. Thus Assumption \ref{assup_fconvex} on the oil friction sounds reasonable from practical experiences.

Even if the function $f$ is convex, the constraints (\ref{constr_nonlinear}) are still nonconvex since they are equality constraints. By relaxing these equalities to inequalities,
\begin{equation}
	\label{constr_newnonlinear} F_{jr} \geq f\left(T_{ave}^{P_{jr}}, Q_{jr}, D_{jr}\right)L_{jr},~ \ \ j=1,...,N^S-1,~ \ \ r=1,...,N^{P}_j,
\end{equation}
we obtain the following convex relaxation of (HOP).
	\begin{equation*}
	\text{(HOPnr2($\underline{z}$, $\overline{z}$))}~~ \ \
	\begin{aligned}
	\min_{\Psi} \quad & C(x, \Delta H^{SP}, \Delta T) \\
	\text{s.t.} ~~~  & (2), (4)-(10), (13)-(21).
	\end{aligned}
	\end{equation*}

The above problem is simply called as (HOPnr2) as before if the bounds of $z$ are not specified.
Since the nonlinear equality constraints become inequality constraints after the convex relaxation, the following remark is slightly different from the ones for (HOP) and (HOPnr1).
\begin{remark}\label{remark_s2}
	Each choice of $s \coloneqq (z, \Delta H^{SP}, \Delta T, H_{out}^{S}, F)$ defines a unique scheme $\Psi$ of (HOPnr2) satisfying all equality constraints. Specifically, $s$ defines a unique feasible scheme $\Psi$ of (HOPnr2) if $\Psi$ satisfies all inequality constraints.
\end{remark}

Since (HOPnr2) is a convex NLP problem, its any local minimizer is also globally optimal. Thus the corresponding objective function value can be used as a lower bound of (HOP). This indicates that (HOPnr2) is suitable for the relaxations of subproblems in the branch-and-bound tree.

However, since (HOPnr2) relaxes constraints \eqref{constr_nonlinear}, it is unfortunate that not only a property similar to Proposition \ref{prop_upperbound} does not hold for the new relaxation, but also a feasible $s$ of (HOPnr2) satisfying the constraints \eqref{constr_xbound} and \eqref{constr_ybound} may not be feasible to (HOP). That is, upper bounds of (HOP) are usually unachievable via the solution of (HOPnr2).
Surprisingly, the convex relaxation (HOPnr2) is equivalent to the nonconvex one (HOPnr1) under some conditions, as shown in the next subsection.


\subsection{The Equivalence of Two Relaxations}

To establish the equivalence between the nonconvex relaxation (HOPnr1) and the convex relaxation (HOPnr2), we first give the following lemmas. They illustrate the order preservation of the temperature and the conditions, which lay foundation for the equivalence theorem.

\begin{lemma}\label{lem_monoT}
	Suppose $\check{s}$ and $\hat{s}$ are feasible to (HOPnr2). For any $j=1,...,N^S-1$, if $\check{T}_{out}^{S_j} \geq \hat{T}_{out}^{S_j}$, we have that
	\begin{equation*}
		\check{T}_{ave}^{P_{jr}} \geq \hat{T}_{ave}^{P_{jr}},~ \ \ r=0,...,N^P_j.
	\end{equation*}
\end{lemma}

\begin{lemma}\label{lem_sameopt}
Suppose Assumption \ref{assup_fconvex} holds. For each $j=1,...,N^S-1$, if there exists a feasible solution  $\tilde{s} = (\tilde{z}, \Delta \tilde{H}^{SP}, \Delta \tilde{T}, \tilde{H}_{out}^{S})$ of (HOPnr1) such that the $\tilde T_{out}^S$ defined by $\tilde{s}$ satisfies
\begin{equation*}
	\tilde{T}_{out}^{S_j} = \overline{T}_{out}^{S_j},
\end{equation*}
then (HOPnr2) is feasible. Moreover, for each feasible solution  $\check{s} = (\check{z}, \Delta \check{H}^{SP}, \Delta \check{T}, \check{H}_{out}^{S}, \check{F})$ of (HOPnr2), there exists a feasible solution $\hat{s} = (\hat{z}, \Delta \hat{H}^{SP}, \Delta \hat{T}, \hat{H}_{out}^{S})$ of (HOPnr1) such that
\begin{equation*}
	C(\check{x}, \Delta \check{H}^{SP}, \Delta \check{T}) \geq C(\hat{x}, \Delta \hat{H}^{SP}, \Delta \hat{T}).
\end{equation*}
\end{lemma}

\begin{theorem}\label{thm_equivalence}
	Under the conditions of Lemma \ref{lem_sameopt}, (HOPnr1) and (HOPnr2) have the same optimal objective value.
\end{theorem}

Noting that since (HOPnr1) is a relaxation of (HOP), if (HOPnr1) does not satisfy the conditions in Lemma \ref{lem_sameopt}, neither does the original problem (HOP). It means that the upper bounds of $T_{out}^{S_j},~ j=1,...,N^S-1$ are loose. In this case, the upper bounds can be tightened without reducing the feasible region of (HOP).
Once the conditions in Lemma \ref{lem_sameopt} are satisfied, due to Theorem \ref{thm_equivalence}, we achieve the global optimal solution of (HOPnr1) by solving (HOPnr2) locally. Such a solution leads to an upper bound of (HOP) according to Proposition \ref{prop_upperbound}. Meanwhile, since (HOPnr2) is a relaxation of (HOP), a lower bound of (HOP) can be obtained by the global optimal solution of (HOPnr2). Therefore solving (HOPnr2) of each subproblem in the branch-and-bound tree enables us to get lower and upper bounds simultaneously.

In fact, the conditions in Lemma \ref{lem_sameopt} can be achieved by a preprocessing procedure on (HOPnr1). For each $j=1,...,N^S-1$, if a feasible solution $\tilde{s}$ of (HOPnr1) is obtained such that
\begin{equation*}
	\{s \mid T_{out}^{S_j} > \tilde{T}_{out}^{S_j},~ s ~\text{is feasible to (HOPnr1)}~\} = \emptyset,
\end{equation*}
then the upper bound of variable $T_{out}^{S_j}$ can be tightened to $\tilde{T}_{out}^{S_j}$, that is
\begin{equation*}
\overline{T}_{out}^{S_j} \coloneqq \min \{\overline{T}_{out}^{S_j}, \tilde{T}_{out}^{S_j}\}.
\end{equation*}
The following example of two stations illustrates the basic idea to find such $\tilde{s}$.

\begin{example}\label{eg_preprocess}
Suppose there are two stations $A$ and $B$. The heads of the oil between $A$ and $B$ must be restricted in $[H_{lb},H_{ub}]$. The outlet temperature at station $A$, denoted as $T_A$, is required to be not higher than 60. However, the inlet temperature at station $B$ is not limited. Our target is to figure out a scheme $(H_A, T_A)$, where $H_A$ is the outlet head at station $A$, so as to satisfy the restriction on heads and maximize $T_A$. To achieve the goal, we can adopt the following procedure.
\begin{enumerate}[(a)]
	\item Initialize $(H_A, T_A) \coloneqq (H_{lb}, 60)$, calculate the heads at each point between stations $A$ and $B$;
	\item Find out the maximal violation point to $H_{lb}$, increase $H_A$ until the head at this point satisfies the lower bound constraint;
	\item Find out the maximal violation point to $H_{ub}$, decrease $T_A$ to some appropriate value, reset $H_A \coloneqq H_{lb}$ and update heads in the whole pipeline;
	\item Repeat steps (b) and (c) until there exists no violation point, return $(H_A, T_A)$.
\end{enumerate}

	\begin{figure}[!h]
		\FIGURE
		{\includegraphics[width=0.6\textwidth]{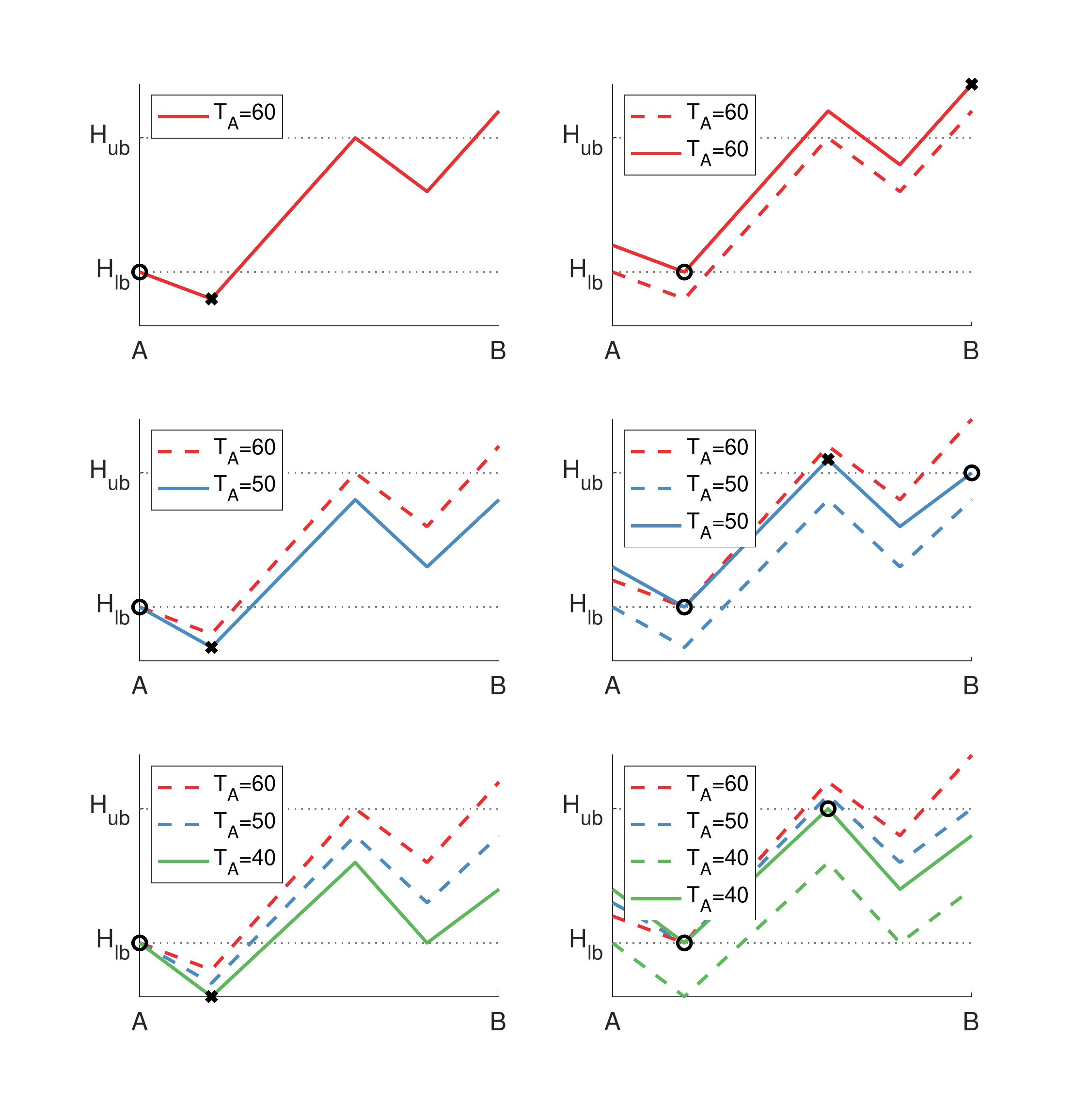}}
		{The Heads of Oil Flow Between the Stations $A$ and $B$\label{example_preprocess}}
		{}
	\end{figure}
	Fig. \ref{example_preprocess} shows the above procedure. As we can see, the final head curve we obtain is such that there are two points reaching the lower bound and upper bound of head, respectively. In fact, any schemes with $T_A > 40$ are infeasible. Either of these two points will violate the bound constraints of heads since the difference of them will be larger than $H_{ub}-H_{lb}$ in that case.
\end{example}

The above example only treats two stations. A detailed preprocessing procedure will be presented in the next section to deal with the case of three or more stations and the infeasibility detected during the tightening.

\section{The Branch-and-Bound Algorithm and Solution Techniques }\label{sec_Algorithms}

The analysis in the former sections enables us to design a branch-and-bound algorithm for solving the problem (HOP) (see subsection \ref{subsec_branch-and-bound}). A detailed preprocessing procedure is described in subsection \ref{subsec_preprocess}. An outer approximation
algorithm is proposed in subsection \ref{outer_approx} for solving (HOPnr2). A warm start strategy is provided in subsection \ref{subsec_warmstart}.

\subsection{The Branch-and-Bound Algorithm}\label{subsec_branch-and-bound}

In our algorithm, the branch-and-bound tree is generated by branching on integer variables $x_j$ and $y_j$, $j=1,...,N^S-1$. For each subproblem in the tree, we solve the corresponding (HOPnr2) to obtain a lower bound of the subproblem. Besides, the solution of (HOPnr2) also leads to a feasible solution of (HOPnr1) through Lemma \ref{lem_sameopt}, which can further be lifted to a feasible solution of the original problem (HOP) according to Proposition \ref{prop_upperbound}.
If the values of $x$ and $y$ are vectors consisting of integers, we have found the global optimal solution of the subproblem so that it can be pruned. Otherwise, we need to branch on one of the fractional variables $x_j$ or $y_j$ to divide the feasible region. After adding two new subproblems, we continue with an unprocessed node in the branch-and-bound tree.
The detailed algorithm is introduced by Algorithm \ref{alg_branch-and-bound}.

\begin{algorithm}[h!]

	\TableSpaced
	{\begin{algorithmic}[1]
		\STATE Initialize $\underline{z} \coloneqq 0$, $\overline{z} \coloneqq (N^{CP}, N^{SP})$, $LB \coloneqq -\infty$, $GLB \coloneqq -\infty,~ GUB \coloneqq +\infty,~ \mathcal{Q} \coloneqq \left\{(\underline{z}, \overline{z}, LB)\right\}$;
		\WHILE{$\mathcal{Q}\neq \emptyset$ and $GLB < GUB$}
			\STATE Choose $q\in\mathcal{Q}$, $\mathcal{Q} \coloneqq \mathcal{Q} \setminus q$, update $GLB,\underline{z},\overline{z},LB$ by $q$; \label{algstep_choosesub}
			\STATE Preprocess (HOPnr1($\underline{z}$, $\overline{z}$)) to check its feasibility and meet the conditions in Lemma \ref{lem_sameopt};
			\IF{$LB < GUB$ and (HOPnr1($\underline{z}$, $\overline{z}$)) is feasible}
				\STATE Solve (HOPnr2($\underline{z}$, $\overline{z}$)) and get the local minimizer $\check{s}  \coloneqq (\check{z},\Delta \check{H}^{SP}, \Delta \check{T}, \check{H}^S_{out}, \check{F})$, update $LB \coloneqq C(\check{x}, \Delta \check{H}^{SP}, \Delta \check{T})$, get the (HOPnr1($\underline{z}$, $\overline{z}$)) feasible solution $\hat{s}$ with $\check{s}$ based on Lemma \ref{lem_sameopt};\label{algstep_solvesub}
				\STATE Obtain $\tilde{s} \coloneqq (\tilde{z},\Delta \tilde{H}^{SP}, \Delta \tilde{T}, \tilde{H}^S_{out})$ with $\hat{s}$ based on Proposition \ref{prop_upperbound};
				\IF{$C(\tilde{x}, \Delta \tilde{H}^{SP}, \Delta \tilde{T}) < GUB$}
					\STATE $GUB \coloneqq C(\tilde{x}, \Delta \tilde{H}^{SP}, \Delta \tilde{T})$, $s^* \coloneqq \tilde{s}$
				\ENDIF
				\IF{there exists $j_0\in\{1,...,N^S-1\}$ such that $\check{x}_{j_0}\notin\mathbb{Z}$ ($\check{y}_{j_0}\notin\mathbb{Z}$)}
					\STATE Let
					\begin{align*}
						Lz &\coloneqq \{z\in\mathbb{Z}^{2N^S-2}_+ \mid x_j = \overline{x}_j,~ y_j = \overline{y}_j,~\text{for all}~j\neq j_0,~ x_{j_0} = \lfloor \check{x}_{j_0} \rfloor (y_{j_0} = \lfloor \check{y}_{j_0} \rfloor)\},\\
						Rz &\coloneqq \{z\in\mathbb{Z}^{2N^S-2}_+ \mid x_j = \underline{x}_j,~ y_j = \underline{y}_j,~\text{for all}~j\neq j_0,~ x_{j_0} = \lceil \check{x}_{j_0} \rceil (y_{j_0} = \lceil \check{y}_{j_0} \rceil)\},
					\end{align*}
					and update $\mathcal{Q}\coloneqq\mathcal{Q}\cup\{(\underline{z},Lz,LB),(Rz,\overline{z},LB)\}$;
				\ENDIF
			\ENDIF
			\STATE Update $GLB$ by
			\begin{equation*}
			GLB \coloneqq \min_{q\in\mathcal{Q}}\{LB \mid LB ~\text{is the lower bound of}~ q\}
			\end{equation*}
		\ENDWHILE
		\IF{$GUB = +\infty$}
			\RETURN the (HOP) problem is infeasible;
		\ELSE
			\RETURN current incumbent solution $s^*$ and $GUB$;
		\ENDIF
	\end{algorithmic}}
	\caption{The Branch-and-Bound Algorithm for Solving (HOP)}
	\label{alg_branch-and-bound}
\end{algorithm}

The following theorems show the finite termination and global property of Algorithm \ref{alg_branch-and-bound}.
\begin{theorem}\label{thm_finite}
	For any inputs of the problem (HOP), Algorithm \ref{alg_branch-and-bound} terminates finitely.
\end{theorem}

\begin{theorem}\label{thm_global}
	Suppose the problem (HOP) is feasible. Then Algorithm \ref{alg_branch-and-bound} returns the global optimal solution of the problem (HOP)
 when it terminates.
\end{theorem}

\subsection{A Preprocessing Algorithm}\label{subsec_preprocess}

To complete the step in line 4 of Algorithm \ref{alg_branch-and-bound}, we need to design a preprocessing procedure for strengthening the bound constraint (\ref{constr_boundTout}) and checking the feasibility of (HOPnr1($\underline{z}$, $\overline{z}$)). In this subsection, we present this preprocessing procedure. Note that based on the constraints (\ref{constr_inoutT}) and (\ref{constr_temprisebound}), the feasible upper bound of the outlet temperature in each station will not be affected by the inlet temperature which is decided by the former stations. A similar situation can be found in the feasible lower bound of the outlet head in each station. So the idea is to individually deal with each station from the one before the last one to the first one. For each of them, like the procedure shown in Example \ref{eg_preprocess}, we manage to obtain a running scheme of pumps and furnaces which can satisfy all the constraints about the station and the pipe right behind it. Such a scheme will reach the outlet head as low as possible and the outlet temperature as high as possible. Since we need to guarantee the uniqueness of the solution of the nonlinear equation, the following assumption is essential.
\begin{assumption}\label{assump_strict}
	Given $Q_{jr} > 0$ and $D_{jr} > 0$, the function $f$ is convex and
	strictly monotonically decreasing about $T_{ave}^{P_{jr}} > 0$.
\end{assumption}
To illustrate the algorithm conveniently, we denote
\begin{equation*}
\underline{H}_{out}^{P_{j0}} \coloneqq \underline{H}_{out}^{S_j},~ \ \ \overline{H}_{out}^{P_{j0}} \coloneqq \overline{H}_{out}^{S_j},~ \ \ \underline{H}_{out}^{P_{jN^P_j}} \coloneqq \underline{H}_{in}^{S_{j+1}},~ \ \ \overline{H}_{out}^{P_{jN^P_j}} \coloneqq \overline{H}_{in}^{S_{j+1}},~ \ \ j=1,...,N^S-1.
\end{equation*}
The preprocessing procedure of (HOPnr1) is introduced in Algorithm \ref{alg_preprocess} in detail.

\begin{algorithm}[!h]
	\TableSpaced
	{\begin{algorithmic}[1]
		\FOR{$j=N^S-1,N^S-2,...,1$}
			\STATE Execute domain propagation on (\ref{constr_outTpart}), (\ref{constr_connectT}), (\ref{constr_connectH}), (\ref{constr_boundTout}) with $j$ fixed and (\ref{constr_boundHin}), (\ref{constr_boundTin}) with $j+1$ fixed;\label{algstep_domainpropa}
			\STATE Let $T_{out}^{S_j} \coloneqq \overline{T}_{out}^{S_j}$, calculate $T_{out}^{P_{jr}},T_{ave}^{P_{jr}},F_{jr},r=1,...,N^P_j$ and $T_{in}^{S_{j+1}}$ by (\ref{constr_nonlinear}),(\ref{constr_outTpart}),(\ref{constr_aveTpart}) and (\ref{constr_connectT});\label{algstep_fixT}
			
			\STATE Let $H_{out}^{S_j} \coloneqq \underline{H}_{out}^{S_j}$, calculate $H_{out}^{P_{jr}},r=0,...,N^P_j$ by (\ref{constr_hchangepart}) and (\ref{constr_connectH});\label{algstep_fixH}
			\STATE Let $r_0 \coloneqq \argmax_{r=0,...,N^P_j}\{\underline{H}_{out}^{P_{jr}}-H_{out}^{P_{jr}}\}$, $\delta H \coloneqq \underline{H}_{out}^{P_{jr_0}}-H_{out}^{P_{jr_0}}$, update $H_{out}^{P_{jr}} \coloneqq H_{out}^{P_{jr}} + \delta H$;\label{algstep_liftH}
			\IF{there exists $r_1\in\{0,...,N^P_j\}$ such that $H_{out}^{P_{jr_1}} > \overline{H}_{out}^{P_{jr_1}}$}\label{algstep_checkHub}
				\IF{$r_1 < r_0$}
					\RETURN (HOPnr1($\underline{z}$, $\overline{z}$)) is infeasible;\label{algstep_infeasible1}
				\ENDIF
					\STATE Solve the following nonlinear equation
					\begin{equation}\label{nonlineareqn}
						\begin{aligned}
						\overline{H}_{out}^{P_{jr_1}} &= \underline{H}_{out}^{P_{jr_0}} - \sum_{t=1}^{r_1-r_0}\left[f(\phi_{jt}^{r_0} u + \psi_{jt}^{r_0}, Q_{j,r_0+t}, D_{j,r_0+t})L_{j,r_0+t} + \Delta Z_{j,r_0+t}\right],
						\end{aligned}
					\end{equation}
					where $u\in\mathbb{R}$ is a variable, $\phi_{jr}^{r_0}$ and $\psi_{jr}^{r_0}$ are parameters.
					
					\IF{the nonlinear equation (\ref{nonlineareqn}) is infeasible}
						\RETURN (HOPnr1($\underline{z}$, $\overline{z}$)) is infeasible;\label{algstep_infeasible2}
					\ENDIF
						\STATE Let $$\overline{T}_{out}^{S_j} \coloneqq \sum_{t=1}^{r_0}\left(T_g^{P_{jt}}+T_f^{P_{jt}}\right)\left[1-\exp(\alpha_{jt}L_{jt})\right]\prod_{k=1}^{t-1}\exp(\alpha_{jk}L_{jk}) + \prod_{t=1}^{r_0}\exp(\alpha_{jt}L_{jt})u.$$ Here $u$ is the solution of \eqref{nonlineareqn}, go to line \ref{algstep_fixT};\label{algstep_tightT}
			\ENDIF
				\STATE Strengthen $\underline{H}_{out}^{S_j}$, $\overline{T}_{out}^{S_j}$ and $\underline{H}_{in}^{S_j}$ via
				\begin{align*}
					\underline{H}_{out}^{S_j} \coloneqq  H_{out}^{S_j},~
					  \underline{H}_{in}^{S_j} \coloneqq \max\{\underline{H}_{in}^{S_j}, \underline{H}_{out}^{S_j}-\overline{x}_jH^{CP}_j-\overline{y}_j\overline{H}^{SP}_j\},~
					  \overline{T}_{in}^{S_j} \coloneqq \min\{\overline{T}_{in}^{S_j},\overline{T}_{out}^{S_j}\};
				\end{align*}\label{algstep_tightH}
				\IF{bound constraints (\ref{constr_boundHin})-(\ref{constr_transitionHead}) are unreasonable}\label{algstep_unreasonable}\label{algstep_checkreasonable}
					\RETURN (HOPnr1($\underline{z}$, $\overline{z}$)) is infeasible;\label{algstep_infeasible3}
				\ENDIF
		\ENDFOR
		\RETURN (HOPnr1($\underline{z}$, $\overline{z}$)) is feasible;
	\end{algorithmic}}
	\caption{The Preprocessing Algorithm for (HOPnr1($\underline{z}$, $\overline{z}$))}
	\label{alg_preprocess}
\end{algorithm}

Domain propagation in step \ref{algstep_domainpropa} is the process that shrinks the bounds of variables $H_{out}^{S_j}$, $H_{in}^{S_{j+1}}$, $T_{out}^{S_j}$ and $T_{in}^{S_{j+1}}$ without affecting the feasible region. In the nonlinear equation \eqref{nonlineareqn}, for all $r=1,...,r_1-r_0$, the parameters $\phi_{jr}^{r_0}$ and $\psi_{jr}^{r_0}$ are defined as
\begin{align*}
\phi_{jr}^{r_0} = &~\frac{1}{3}\left[1+2\exp\left(-\alpha_{j,r_0+r}L_{j,r_0+r}\right)\right]\exp\left(-\sum_{t=1}^{r-1}\alpha_{j,r_0+t}L_{j,r_0+t}\right), \\
\psi_{jr}^{r_0} = &~\frac{1}{3}\left[1+2\exp\left(-\alpha_{j,r_0+r}L_{j,r_0+r}\right)\right]\sum_{t=1}^{r-1}\left(T_g^{P_{j,r_0+t}}+T_f^{P_{j,r_0+t}}\right)\left[1-\exp\left(-\alpha_{j,r_0+t}L_{j,r_0+t}\right)\right] + \\ &~\frac{2}{3}\left(T_g^{P_{j,r_0+r}}+T_f^{P_{j,r_0+r}}\right)\left[1-\exp\left(-\alpha_{j,r_0+r}L_{j,r_0+r}\right)\right].
\end{align*}
``Unreasonable'' in step \ref{algstep_unreasonable} means that the lower bound of some variable is strictly larger than its upper bound. We give the following lemmas to illustrate the meaning of the nonlinear equation \eqref{nonlineareqn} and the infeasibility certificate in step \ref{algstep_infeasible1}.

\begin{lemma}\label{lem_u0upper}
	Suppose Assumption \ref{assup_fconvex} holds and the nonlinear equation \eqref{nonlineareqn} has a feasible solution $\tilde{u}$. Then $\tilde{u}$ is an upper bound of variable $T_{out}^{P_{jr_0}}$. Moreover, the variable $T_{out}^{S_j}$ has an upper bound
	\begin{equation*}
		\sum_{t=1}^{r_0}\left(T_g^{P_{jt}}+T_f^{P_{jt}}\right)\left[1-\exp(\alpha_{jt}L_{jt})\right]\prod_{k=1}^{t-1}\exp(\alpha_{jk}L_{jk}) + \prod_{t=1}^{r_0}\exp(\alpha_{jt}L_{jt})\tilde{u}.
	\end{equation*}
\end{lemma}

\begin{lemma}\label{lem_nleqfeasible}
	Suppose Assumption \ref{assump_strict} holds and (HOPnr1($\underline{z}$, $\overline{z}$)) is feasible. For all $0 \leq r_0 < r_1 \leq N^P_j$, if there exists $\hat{u}$ such that
	\begin{equation*}
	\begin{aligned}
	\overline{H}_{out}^{P_{jr_1}} &< \underline{H}_{out}^{P_{jr_0}} - \sum_{t=1}^{r_1-r_0}\left[f(\phi_{jt}^{r_0} \hat{u}+\psi_{jt}^{r_0}, Q_{j,r_0+t}, D_{j,r_0+t})L_{j,r_0+t} + \Delta Z_{j,r_0+t}\right], \\
	\end{aligned}
	\end{equation*}
	then the nonlinear equation \eqref{nonlineareqn} has a unique solution. Furthermore, if $\tilde{u}$ is feasible to \eqref{nonlineareqn}, then the following property holds
	\begin{equation*}
		\overline{T}_{out}^{S_j} > \sum_{t=1}^{r_0}\left(T_g^{P_{jt}}+T_f^{P_{jt}}\right)\left[1-\exp(\alpha_{jt}L_{jt})\right]\prod_{k=1}^{t-1}\exp(\alpha_{jk}L_{jk}) + \prod_{t=1}^{r_0}\exp(\alpha_{jt}L_{jt})\tilde{u}.
	\end{equation*}
\end{lemma}

\begin{lemma}\label{lem_r0>r1}
	Suppose Assumption \ref{assup_fconvex} holds. For all $0 \leq r_1 < r_0 \leq N^P_j$, if there exists some $\hat{u}$ such that
	\begin{equation*}
	\begin{aligned}
	\underline{H}_{out}^{P_{jr_0}} &> \overline{H}_{out}^{P_{jr_1}} - \sum_{t=1}^{r_0-r_1}\left[f(\phi_{jt}^{r_1} \hat{u}+\psi_{jt}^{r_1}, Q_{j,r_1+t}, D_{j,r_1+t})L_{j,r_1+t} + \Delta Z_{j,r_1+t} \right], \\
	\overline{T}_{out}^{S_j} &= \sum_{t=1}^{r_1}\left(T_g^{P_{jt}}+T_f^{P_{jt}}\right)\left[1-\exp(\alpha_{jt}L_{jt})\right]\prod_{k=1}^{t-1}\exp(\alpha_{jk}L_{jk}) + \prod_{t=1}^{r_1}\exp(\alpha_{jt}L_{jt})\hat{u},
	\end{aligned}
	\end{equation*}
	then (HOPnr1($\underline{z}$, $\overline{z}$)) is infeasible.
\end{lemma}

With all the above lemmas, we conclude that Algorithm \ref{alg_preprocess} either gives the infeasibility of (HOPnr1) or tightens the bounds of variables to meet the conditions in Lemma \ref{lem_sameopt}.

\begin{theorem}\label{thm_preprocess}
	Algorithm \ref{alg_preprocess} does not affect the feasible region of the problem (HOPnr1($\underline{z}$, $\overline{z}$)). When Algorithm \ref{alg_preprocess} terminates, either there exists a solution $\tilde{s}$ such that
	\begin{equation*}
		\tilde{T}_{out}^{S_j} = \overline{T}_{out}^{S_j},~ \ \ j=1,...,N^S-1,
	\end{equation*}
	or it returns that (HOPnr1($\underline{z}$, $\overline{z}$)) is infeasible.
\end{theorem}

\subsection{An Outer Approximation Algorithm}\label{outer_approx}

Noting that (HOPnr2($\underline{z}$, $\overline{z}$)) is a convex programming problem, the outer approximation method \citep{Kelley1960The} can be introduced to solve it in step \ref{algstep_solvesub} of Algorithm \ref{alg_branch-and-bound}. This method solves (HOPnr2($\underline{z}$, $\overline{z}$)) implicitly by approximating the nonlinear constraints with linear ones, which are defined as:
\begin{equation}\label{constr_linearapprox}
F_{jr} \geq \left[\zeta_{j,r,\omega} T_{ave}^{P_{jr}} + f(\omega, Q_{jr}, D_{jr}) -\zeta_{j,r,\omega}\right]L_{jr},~ \ \ \zeta_{j,r,\omega} \in \partial f(\omega, Q_{jr}, D_{jr}),~ \ \ (j,r,\omega)\in\Omega,
\end{equation}
where $\omega$ is a parameter and $\partial f$ is the set of subgradients of the function $f$. Since these linear constraints are all valid for (\ref{constr_newnonlinear}), the approximation will not affect any feasible solutions. 
The linear approximation problem can be defined as the following linear programming (LP) problem
\begin{equation*}
	\text{(HOPlr($\underline{z}$, $\overline{z}$, $\Omega$))}~~ \ \
	\begin{aligned}
	\min_{\Psi} \quad & C(x, \Delta H^{SP}, \Delta T) \\
	\text{s.t.} ~~~  & (2), (4)-(10), (13)-(21), (23).
	\end{aligned}
\end{equation*}
\begin{remark}\label{remark_s3}
	Each choice of $s \coloneqq (z, \Delta H^{SP}, \Delta T, H_{out}^{S}, F)$ defines a unique scheme $\Psi$ of (HOPlr) satisfying all equality constraints. Specifically, $s$ defines a unique feasible scheme $\Psi$ of (HOPlr) if $\Psi$ satisfies all inequality constraints.
\end{remark}

The basic idea of our outer approximation algorithm is as follows. At first, we solve (HOPlr($\underline{z}$, $\overline{z}$, $\Omega$)) with some initial index set $\Omega_0$ and obtain its optimal solution. If this solution satisfies the constraints (\ref{constr_newnonlinear}), it is also feasible to (HOPnr2($\underline{z}$, $\overline{z}$)). Otherwise, we add a linear outer approximation constraint to (HOPlr($\underline{z}$, $\overline{z}$, $\Omega$)) at the maximal violation point so that the updated problem will not violate (\ref{constr_newnonlinear}) with the same solution. Then we update the index set $\Omega$ iteratively until the optimal solution of (HOPlr($\underline{z}$, $\overline{z}$, $\Omega$)) slightly violates the constraints of (HOPnr2($\underline{z}$, $\overline{z}$)) under some tolerance. Due that (HOPlr($\underline{z}$, $\overline{z},\Omega$)) is a relaxation of (HOPnr2($\underline{z}$, $\overline{z}$)), such a solution is approximately optimal for the latter problem. More details are presented in Algorithm \ref{alg_outerapprox}.

\begin{algorithm}
	\TableSpaced
	{\begin{algorithmic}[1]
		\STATE Initialize $\Omega \coloneqq \Omega_0$, violation tolerance $\epsilon > 0$;
		\STATE Solve (HOPlr($\underline{z}$, $\overline{z}$, $\Omega$)) and get the optimal solution $\check{s}$, calculate $\hat{F}$ by $\check{T}_{ave}^{P_{jr}}$ and \eqref{constr_nonlinear};\label{algstep_solvelinearsub}
		\STATE Find the maximal violation $vio_{max}$ and the corresponding index $j_0$, $r_0$ with
		\begin{align*}
			(j_0,r_0) &\coloneqq \mathop{\arg\max}_{j=1,...,N^S-1 \atop r=1,...,N^P_j}\left\{\hat{F}_{jr} - \check{F}_{jr}\right\}, \\
			vio_{max} &\coloneqq \hat{F}_{j_0r_0} - \check{F}_{j_0r_0};
		\end{align*}
		\IF{$vio_{max} > \epsilon$}
			\STATE Update $\Omega \coloneqq \Omega \cup \{(j_0,r_0,\check{T}_{ave}^{P_{j_0r_0}})\}$, go to line \ref{algstep_solvelinearsub};
		\ELSE
			\RETURN $\check{s}$ is the approximate optimal solution of (HOPnr2($\underline{z}$, $\overline{z}$))
		\ENDIF
	\end{algorithmic}}
	\caption{The Outer Approximation Algorithm for (HOPnr2($\underline{z}$, $\overline{z}$))}
	\label{alg_outerapprox}
\end{algorithm}

For Algorithm \ref{alg_outerapprox}, the initial index set $\Omega_0$ can be set to the empty set. Furthermore, when embedding Algorithm \ref{alg_outerapprox} into the branch-and-bound algorithm, Algorithm \ref{alg_branch-and-bound}, we may develop some warm start strategy in
choosing $\Omega_0$. See the next subsection for more details.

\subsection{A Warm Start Strategy}\label{subsec_warmstart}

To further speed up Algorithm \ref{alg_branch-and-bound}, we consider a warm start strategy for Algorithm \ref{alg_outerapprox}. For each subproblem (HOP($\underline{z}$, $\overline{z}$)), we call Algorithm \ref{alg_outerapprox} in step \ref{algstep_solvesub} of Algorithm \ref{alg_branch-and-bound} to solve the nonlinear relaxation (HOPnr2($\underline{z}$, $\overline{z}$)). Note that the constraints \eqref{constr_linearapprox} do not contain the branching variables $x$ and $y$. When Algorithm \ref{alg_outerapprox} stops, the constraints \eqref{constr_linearapprox} defined by the final value of $\Omega$ are actually valid not only for the (HOPnr2($\underline{z}$, $\overline{z}$)), but also for the (HOPnr2) relaxed from any other subproblems. This indicates that the initial value $\Omega_0$ can inherit the final value of $\Omega$ in the last call of Algorithm \ref{alg_outerapprox}. Such inheritances can be viewed as a warm start strategy of the outer approximation solving procedure of (HOPnr2($\underline{z}$, $\overline{z}$)).

On one hand, more constraints in (HOPlr($\underline{z}$, $\overline{z}$, $\Omega$)) lead to higher computational cost at each iteration. On the other hand, the warm start strategy provides sufficiently good initial sets $\Omega_0$ for each call (except for the first call) of Algorithm \ref{alg_outerapprox}, which may probably reduce the iteration of this algorithm. Our numerical results show that the performance of
Algorithm \ref{alg_branch-and-bound} with the warm start strategy is much better.

\section{Numerical Results}\label{sec_Numerical_results}

In this section, we illustrate the practical validity of Algorithm \ref{alg_branch-and-bound} through numerical tests with a real world HOP problem. All tests in this section were done on a MacBook Pro laptop with Core i7 CPU and 16 GB RAM. Codes were run on the MATLAB R2018a platform.


Our test object is the Q-T heated oil pipeline located in China. It is 548.54 kilometers long in total. There are nine stations in the whole pipeline. The elevation map is given in Fig. \ref{fig_qt_mileele}. The grey solid line shows the original elevation map of the Q-T pipeline. Stations are marked through red stars. For the convenience of computation, the blue line in Fig. \ref{fig_qt_mileele} is implemented as the elevation data in our tests. Further, we divide the pipeline into 131 segments. Each of them is no longer than five kilometers. The split points of the pipeline are marked with blue circles.

\begin{figure}[!h]
	\FIGURE
	{\includegraphics[width=1.0\textwidth]{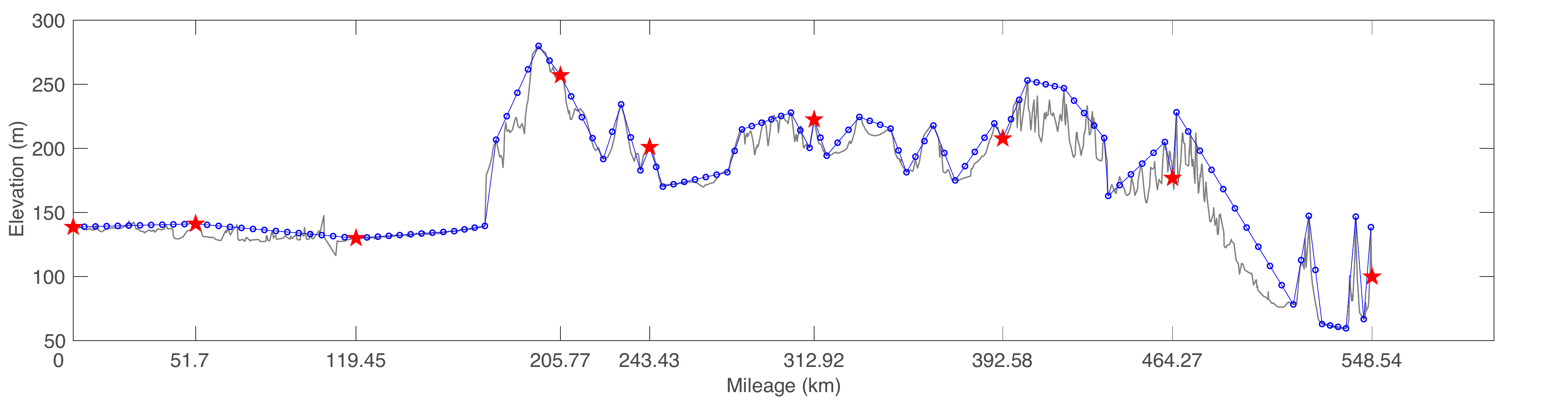}}
	{Elevation Map of the Q-T Pipeline. \label{fig_qt_mileele}}
	{}
\end{figure}

For the oil transported in the Q-T pipeline, its density and specific heat are 859 kg/m$^3$ and 2,400 J/(kg$\cdot^\circ$C), respectively. It has a relatively high freezing point, which is 32$^\circ$C, and high viscosity. Data of the dynamic viscosity of the oil are shown in Table \ref{table_viscosity}.

\begin{table}[!h]
	\TABLE
	{Dynamic Viscosity of the Oil Transported in Q-T Pipeline. \label{table_viscosity}}
	  {\begin{tabular}{rr}
	  \toprule
	  Temperature ($^\circ$C) & Dynamic viscosity (mPa$\cdot$s) \\
	  \midrule
	  35    & 107.4 \\
	  36    & 82.6 \\
	  37    & 64.1 \\
	  38    & 55.7 \\
	  39    & 48.3 \\
	  40    & 41.9 \\
	  42    & 27.6 \\
	  46    & 21.9 \\
	  50    & 19.2 \\
	  54    & 16.9 \\
	  60    & 13.7 \\
	  68    & 10.7 \\
	  \bottomrule
	  \end{tabular}}%
	  {}
  \end{table}%

By curve fitting on the above viscosity data, the kinematic viscosity $\nu$ is obtained
\begin{equation*}
	\nu(T) = \left[8.166\times 10^6\exp(-0.3302T)+77.04\exp(-0.02882T)\right]/1000/\rho.
\end{equation*}
Fig. \ref{fig_vistemp} shows the curve fitting result. It is easy to verify that $\nu$ is smooth when $T > 0$ and satisfies Assumption \ref{assump_strict}.

\begin{figure}[!h]
	\FIGURE
	{\includegraphics[width=0.6\textwidth]{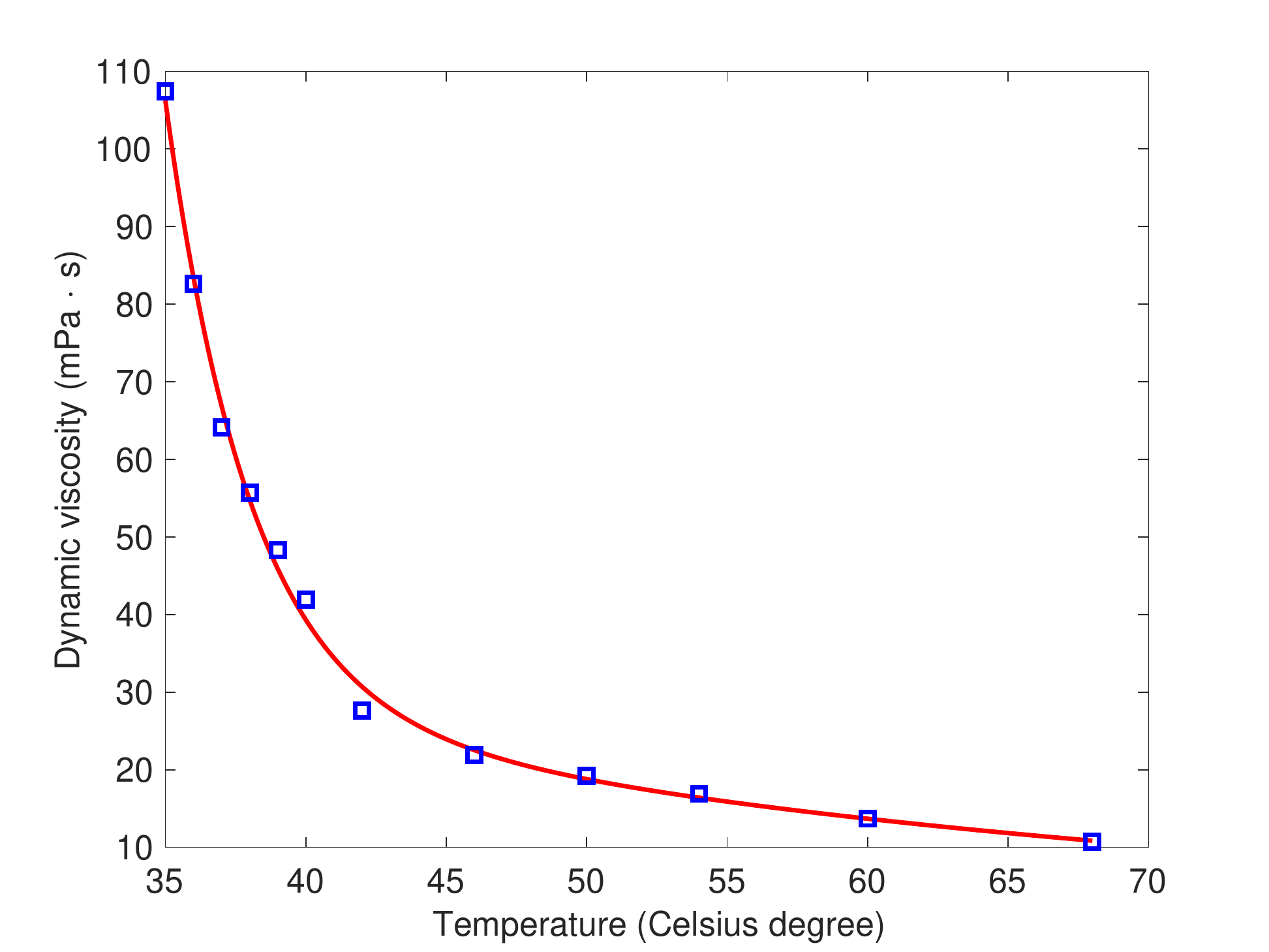}}
	{Viscosity-Temperature Curve of the Oil. \label{fig_vistemp}}
	{}
\end{figure}

For the station and pipe configurations, the inlet head and temperature of the first station are 57.0 meters and 40.7$^\circ$C, respectively. For each station $j$, the bounds of inlet and outlet values are set as follows.
\begin{equation*}
	\underline{H}_{in}^{S_j} = \underline{H}_{out}^{S_j} = 35.6,~ \ \ \overline{H}_{in}^{S_j} = \overline{H}_{out}^{S_j} = 748.4,~ \ \ \underline{T}_{in}^{S_j} = \underline{T}_{out}^{S_j} = 35,~ \ \ \overline{T}_{in}^{S_j} = \overline{T}_{out}^{S_j} = 48.
\end{equation*}
The bounds of $H_{out}^{P_{jr}}$, $r=1,\dots,N^P_j-1,$ are given as the same as $H_{out}^{S_j}$. Other parameters are presented in Table \ref{table_deployments}.
\begin{table}[!h]
	\TABLE
	{Station and Pump Deployments in Q-T Pipeline. \label{table_deployments}}
		{\begin{tabular}{c|r|r|rrr|rrrr}
			\toprule
		Station & $Q_{j0}$ (m$^3$/h) & $D_{j0}$ (mm) & $H^{CP}_j$ (m) & $N^{CP}_j$ & $\xi_{j}^{CP}$ & $\underline{H}^{SP}_j$ (m) & $\overline{H}^{SP}_j$ (m) & $N^{SP}_j$ & $\xi_j^{SP}$ \\
		\midrule
    1     & 2,212  & 740 & 235.76  & 4     & 78.7\% & -& -     & 0     & - \\
    2     & 2,810  & 740 & -     & 0     & -     & - & -     & 0     & - \\
    3     & 2,215  & 772 & 245.70  & 3     & 78.0\% & 94.01 & 244.24 & 1     & 80.3\% \\
    4     & 2,941  & 622 & 222.36  & 4     & 83.5\% & - & -     & 0     & - \\
    5     & 2,751  & 685 & 229.96  & 3     & 83.3\% & 102.84 & 231.35 & 1     & 83.8\% \\
    6     & 2,022  & 695 & 239.56  & 3     & 75.7\% & 95.34 & 239.56 & 1     & 79.5\% \\
    7     & 2,102  & 715 & 237.96  & 3     & 77.0\% & 92.94 & 237.96 & 1     & 79.6\% \\
		8     & 2,047  & 705 & -     & 0     & -     & - & -     & 0     & - \\
		\bottomrule
		\end{tabular}}%
		{}

  %
\end{table}%
Note that the inner diameters given in Table \ref{table_deployments} are the equivalent pipe diameters. With the data of volume flow and inner diameter of the pipe, function $f$ in (HOP) is obtained by formula \eqref{darcy}.

With the above data, we set up the corresponding (HOP) problem and manage to get the global optimal solution through the algorithms proposed in section \ref{sec_Algorithms}. The operation scheme (optimal scheme) based on this solution is compared with a practical scheme. Table \ref{table_schemes} and Fig. \ref{fig_schemes} show the comparison results. In Table \ref{table_schemes}, $C_{power}$ and $C_{fuel}$ are the total costs of pumps and furnaces in each station, respectively. Since we have no head data between each station for the practical scheme, the first pressure curve in Fig. \ref{fig_schemes} only consists of the inlet and outlet values of each station.  We can see that there are considerable differences between the two schemes. The total costs of the whole Q-T pipeline are 542,751.95 yuan per day for the practical scheme, and 505,676.76 yuan per day for the scheme coming from the optimal solution of (HOP). This indicates that the optimal scheme gains a 6.83\% improvement, which proves that the proposed model and the corresponding algorithms can bring significant economic benefits.

\begin{table}[!h]
\TABLE
  {Practical and Optimal Operation Scheme Comparison (Cost Unit: yuan/d). \label{table_schemes}}
		{\begin{tabular}{c|rrrrr|rrrrr}
			\toprule
		\multicolumn{1}{c|}{\multirow{3}[-2]{*}{Station}} & \multicolumn{5}{c|}{Practical scheme} & \multicolumn{5}{c}{Optimal scheme} \\
		\cmidrule{2-11}
					& \multicolumn{1}{c}{$x$} & \multicolumn{1}{c}{$\Delta H^{SP}$} & \multicolumn{1}{c}{$C_{power}$} & \multicolumn{1}{c}{$\Delta T$} & \multicolumn{1}{c|}{$C_{fuel}$} & \multicolumn{1}{c}{$x$} & \multicolumn{1}{c}{$\Delta H^{SP}$} & \multicolumn{1}{c}{$C_{power}$} & \multicolumn{1}{c}{$\Delta T$} & \multicolumn{1}{c}{$C_{fuel}$} \\
					\midrule
    1     & 3     & 0     & 72,510.05 & 3.00  & \Black{14,047.23}  & 3     & 0     & 72,510.05  & 7.30  & 34,181.59 \\
    2     & 0     & 0     & 0     & 3.50  & 20,818.94  & 0     & 0     & 0.00  & 1.24  & \Black{7,391.16} \\
    3     & 1     & 219.96 & \Black{47,578.26} & 2.90  & \Black{13,597.40}  & 1     & 244.24 & 50,020.79  & 10.13  & 47,496.64 \\
    4     & 3     & 0     & 85,751.64 & 2.80  & \Black{17,431.60}  & 1     & 0     & \Black{28,583.88}  & 3.83  & 23,825.27 \\
    5     & 2     & 104.33 & 67,879.53 & 4.10  & 23,875.84  & 1     & 231.35 & \Black{55,390.45}  & 0.00  & \Black{0} \\
    6     & 1     & 170.27 & \Black{39,140.92} & 6.48  & \Black{27,735.78}  & 2     & 224.68 & 67,530.40  & 6.84  & 29,293.14 \\
    7     & 1     & 203.94 & 43,330.03 & 8.80  & 39,156.12  & 0     & 175.05 & \Black{16,858.35}  & 7.46  & \Black{33,178.22} \\
		8     & 0     & 0     & 0     & 6.90  & \Black{29,898.62}  & 0     & 0     & 0.00  & 9.10  & 39,416.83 \\
		\bottomrule
		\end{tabular}}%
		{}
  %
\end{table}%

\begin{figure}
	\FIGURE
	{\includegraphics[width=1\textwidth]{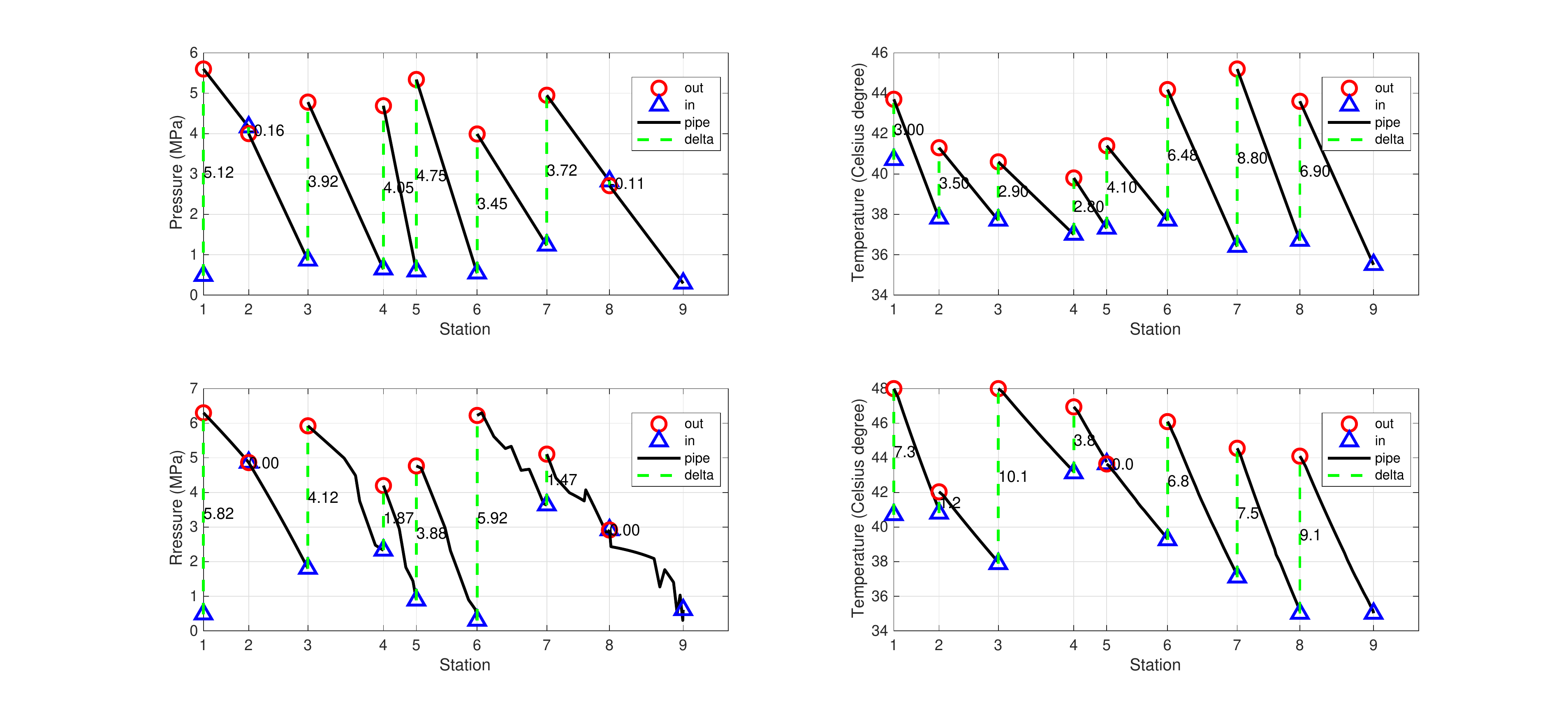}}
	{Pressure and Temperature Curves of the Practical Scheme (Up) and the Optimal Scheme (Down). \label{fig_schemes}}
	{}
\end{figure}

Except for the comparison of solution qualities, different solution procedures were also tested for (HOPnr2). We solved the problem (HOP) of Q-T pipeline with Algorithm \ref{alg_branch-and-bound} using three different routines for (HOPnr2). These methods are listed in Table \ref{table_methods}. The IPM based method solves (HOPnr2) directly, while the outer approximation based methods solve LP at each inner iteration. Table \ref{table_hopnr2} shows the comparison results among them. From Table \ref{table_hopnr2}, we can conclude that although the outer approximation based methods lose the quality of optimal solution, the optimal objective values of the three routines are almost the same. Moreover, \verb|oap+ws| is much faster than the other two methods, which makes the algorithms proposed in this paper become more practical for HOP problems. Notice that although \verb|oap+ws| is much faster than \verb|nlp|,  \verb|oap| is less efficient than \verb|nlp|. This shows the advantage of implementing outer approximation in successive NLP algorithms.

\begin{table}[!h]
	\TABLE
	{Three Routines for Solving (HOPnr2). \label{table_methods}}
	{\begin{tabular}{lll}
		\toprule
		Method & Algorithm & Solver \\
		\midrule
		{\ttfamily nlp} & IPM & {\ttfamily IPOPT} \\
		{\ttfamily oap} & Algorithm \ref{alg_outerapprox} with $\Omega_0 = \emptyset$ & {\ttfamily LINPROG} \\
		{\ttfamily oap+ws} & Algorithm \ref{alg_outerapprox} with $\Omega_0$ inheriting from former subproblems & {\ttfamily LINPROG} \\
		\bottomrule
	\end{tabular}}
	{}
\end{table}

\begin{table}[!h]
	\TABLE
	{Objective Values and Time by Different Routines for (HOPnr2). \label{table_hopnr2}}
	{\begin{tabular}{rr|rr|rr}
		\toprule
		\multicolumn{2}{c|}{{\ttfamily nlp}} & \multicolumn{2}{c|}{{\ttfamily oap}} & \multicolumn{2}{c}{{\ttfamily oap+ws}} \\
		\midrule
		$C$ (yuan/d) & Time (s) & $C$ (yuan/d) & Time (s) & $C$ (yuan/d) & Time (s) \\
		505,676.76 & 40.8 & 505,676.12 & 81.5 & 505,676.33 & \Black{8.6} \\
		\bottomrule
	\end{tabular}}
	{}
\end{table}

\section{Some Concluding Remarks}\label{sec_Conclusions}

In this paper, we proposed an MINLP model for HOP problems and designed an efficient algorithm based on the branch-and-bound framework. As long as the problem (HOP) is feasible, the proposed algorithm is theoretically guaranteed for obtaining the global optimal solution. Furthermore, the algorithm can be accelerated by some outer approximation and warm start approaches. The numerical results with a real world HOP problem showed that the proposed model and algorithm can achieve a more economic scheme compared with the used practical scheme. The high efficiency of the outer approximation based method also guarantees the practicability.

Nevertheless, there may exist more general cases against Assumption \ref{assumption_fund}. If we regard that the friction heat of oil is variable with the friction, then $T_f^{P_{jr}}$ in the constraints \eqref{constr_outTpart} has to be defined as
\begin{equation*}
	T_f^{P_{jr}} = \frac{gF_{jr}\rho Q_{jr}}{K_{jr}\pi d_{jr}},~ \ \ j=1,...,N^S-1,~ \ \ r=1,...,N^{P}_j.
\end{equation*}
Such an extension may bring the infeasibility of both the heads $H_{out}^{P_{jr}}$ and the temperatures $T_{out}^{P_{jr}}$ in some pipe segments after the constraints \eqref{constr_nonlinear} are relaxed to inequalities. Therefore, the equivalence between the extended (HOPnr1) and (HOPnr2) can hardly be guaranteed. We feel that more investigation is required for this general case.
The model for HOP problems can also be generalized to for example the situation where the function $f$ in the constraints \eqref{constr_nonlinear} is replaced with a nonlinear function without Assumption \ref{assup_fconvex} or \ref{assump_strict}, or the situation where there exist nonlinear pumps and furnaces efficiency formulation in the objective function of (HOP). In future, we may consider the extensions of the theories and algorithms of this paper to more general models.



\section*{Acknowledgments}
Many thanks are due to Jianjun Liu from CNPC Key Laboratory of Oil \& Gas Storage and Transportation, PetroChina Pipeline R \& D Center,  who provided great assistances on the background of the HOP problem and testing instance. This research was supported by the Chinese Natural Science Foundation (Nos. 11631013 and 11701137) and the National 973 Program of China (No. 2015CB856002).

\nocite{*}

\bibliographystyle{informs2014}
\bibliography{Global_alg_for_HOP}

\begin{APPENDIX}{Proofs of Propositions, Lemmas and Theorems}
\proof{Proof of Proposition \ref{prop_np}}
We show that (HOP($\underline{z}$, $\overline{z}$)) can be reduced to the cutting stock problem \citep{gilmore1961linear} \citep{gilmore1963linear}, which is $\mathcal{NP}$-hard. Under the following parameter settings:
\begin{itemize}
	\item $\underline{z} = (0,0)$,\ $\overline{z} = (N^{CP}, 0)$,\ $N^{CP}_j \in \mathbb{Z}_+$,\ $H^{CP}_j \in \mathbb{Z}_+$,\ $j=1,...,N^S-1$;
	\item $L_{jr} = 0$,\ $j=1,...,N^S-1$,\ $r=1,...,N^P_j$;
	\item $\Delta Z_{jr} = 0$, $j=1,...,N^S-1$, $r=1,...,N^P_j$;
	\item $\underline{H}_{in}^{S_1} = \overline{H}_{in}^{S_1}$, $\underline{H}_{in}^{S_{N^S}} = \overline{H}_{in}^{S_{N^S}}$, $\underline{H}_{in}^{S_1}-\underline{H}_{in}^{S_{N^S}} \in \left(0,\sum_{j=1}^{N^S-1}N^{CP}_jH^{CP}_j\right) \cap \mathbb{Z}_+$;
	\item $\underline{H}_{in}^{S_j} = 0$, $ \overline{H}_{in}^{S_j} = +\infty$, $j=2,...,N^S-1$, $\underline{H}_{out}^{S_j} = 0$, $ \overline{H}_{out}^{S_j} = +\infty$, $j=1,...,N^S-1$;
	\item $\underline{H}_{out}^{P_{jr}} = 0$, $\underline{H}_{out}^{P_{jr}} = +\infty$, $j=1,...,N^S-1$, $r=1,...,N^P_j$;
	\item $\underline{T}_{in}^{S_j} = \overline{T}_{in}^{S_j}$, $j=1,...,N^S$, $\underline{T}_{out}^{S_j} = \overline{T}_{out}^{S_j}$, $j=1,...,N^S-1$,
\end{itemize}
the original (HOP($\underline{z}$, $\overline{z}$)) reduces to
\begin{equation}\label{reduceHOP}
\begin{aligned}
\min_{x} \quad & \sum_{j=1}^{N^S-1}
C_p\rho Q_{j0} g\frac{H^{CP}_j}{\xi^{CP}_{j}} x_j \\
\text{s.t.} ~~~  & \sum_{j=1}^{N^S-1} x_jH^{CP}_j \geq \underline{H}_{in}^{S_1}-\underline{H}_{in}^{S_{N^S}}, \\
& x_j \in \{0,...,N^{CP}_j\},~ \ \ j=1,...,N^S-1.
\end{aligned}
\end{equation}
Then let $V \coloneqq \underline{H}_{in}^{S_1}-\underline{H}_{in}^{S_{N^S}}$ be the number of rolls of length $l$ demanded, $x_j$ is the number of times the $j$-th cutting pattern is used, $U_j \coloneqq H^{CP}_j$ is the number of rolls of length $l$ produced each time the $j$-th way of cutting up a roll is used, $W_j \coloneqq C_p\rho Q_{j0} g H^{CP}_j/\xi_{j}$ is the cost of the parent roll from which the $j$-th cutting pattern is cut. Then according to \cite{gilmore1963linear}, the problem (\ref{reduceHOP}) can be viewed as a cutting stock problem, which is $\mathcal{NP}$-hard. This completes the proof.
\Halmos
\endproof

	\proof{Proof of Proposition \ref{prop_upperbound}}
	(HOPnr1) is a relaxation of (HOP). Therefore the ``only if'' direction is obvious. We only need to prove that $\hat{s}$ defines a feasible solution of (HOP($\underline{x}$, $\overline{x}$)).
	Suppose $\check{s}$ defines
	\begin{equation*}
		\check{\Psi} \coloneqq \left(\check{z}, \Delta \check{H}^{SP}, \Delta \check{T}, \check{H}_{in}^S, \check{H}_{out}^S, \check{T}_{in}^S, \check{T}_{out}^S, \check{H}_{out}^P, \check{T}_{out}^P, \check{T}_{ave}^P, \check{F}\right),
	\end{equation*}
	that is feasible to (HOPnr1($\underline{x}$, $\overline{x}$)). Note that the differences between (HOP) and (HOPnr1) are the bounds of variables $x$ and $y$. With
	\begin{align*}
		\hat{x_j} &= \lceil \check{x_j} \rceil \in \{\underline{x_j}, \dots, \overline{x_j}\},~ \ \ j=1,\dots,N^S-1, \\
		\hat{y_j} &= \lceil \check{y_j} \rceil \in \{\underline{y_j}, \dots, \overline{y_j}\},~ \ \ j=1,\dots,N^S-1, \\
		\check{H}_{in}^{S_j} &+ \hat{x}_j H^{CP}_j + \Delta \hat{H}^{SP}_j \geq \check{H}_{in}^{S_j} + \check{x}_j H^{CP}_j + \Delta \check{H}^{SP}_j \geq \check{H}_{out}^{S_j} = \hat{H}^{S_j}_{out},~ \ \ j=1,...,N^S-1,
	\end{align*}
	it is easy to verify that
	\begin{equation*}
		\hat{\Psi} \coloneqq \left(\hat{z}, \Delta \hat{H}^{SP}, \Delta \hat{T}, \check{H}_{in}^S, \hat{H}_{out}^S, \check{T}_{in}^S, \check{T}_{out}^S, \check{H}_{out}^P, \check{T}_{out}^P, \check{T}_{ave}^P, \check{F}\right)
	\end{equation*}
	is feasible to (HOP) and $\hat{\Psi}$ is defined by $\hat{s}$. Therefore the proof is completed.
	\Halmos
\endproof

	\proof{Proof of Lemma \ref{lem_monoT}}
	We first prove $\check{T}_{out}^{P_{jr}} \geq \hat{T}_{out}^{P_{jr}},r=0,...,N^P_j,$ by induction. When $r=0$, the conclusion holds according to (\ref{constr_connectH}). Suppose that $\check{T}_{out}^{P_{jr_0}} \geq \hat{T}_{out}^{P_{jr_0}}$.
	Then we show the inequality still holds for $r_0+1$. From (\ref{constr_outTpart}), we have
	\begin{align*}
		\check{T}_{out}^{P_{j,r_0+1}} &= T_g^{P_{j,r_0+1}} + T_f^{P_{j,r_0+1}} + \left[\check{T}_{out}^{P_{jr_0}} - \left(T_g^{P_{j,r_0+1}}+T_f^{P_{j,r_0+1}}\right)\right]e^{-\alpha_{j,r_0+1}L_{j,r_0+1}} \\
		&\geq T_g^{P_{j,r_0+1}} + T_f^{P_{j,r_0+1}} + \left[\hat{T}_{out}^{P_{jr_0}} - \left(T_g^{P_{j,r_0+1}}+T_f^{P_{j,r_0+1}}\right)\right]e^{-\alpha_{j,r_0+1}L_{j,r_0+1}} \\
		&= \hat{T}_{out}^{P_{j,r_0+1}}.
	\end{align*}
	This together with constraint (\ref{constr_aveTpart}) shows that the conclusion of this lemma is true. The proof is completed.
	\Halmos
\endproof

	\proof{Proof of Lemma \ref{lem_sameopt}}
	Since (HOPnr2) is a relaxation of (HOPnr1), which is feasible due to the given condition, the feasibility of (HOPnr2) is obvious. Next, we prove that the $\hat{s}$ defined by
	\begin{align*}
		\hat{z} &= \check{z}, \\
		\Delta \hat{H}^{SP} &= \Delta \check{H}^{SP}, \\
		\Delta \hat{T} &= \Delta \check{T}, \\
		\hat{H}_{out}^{S_j} &= \left\{
			\begin{aligned}
				& \check{H}_{out}^{S_j},
				&& ~\text{if for all}~ r=1,\dots,N^P_j, \\
				&&& \check{F}_{jr} ~\text{and}~ \check{T}_{ave}^{P_{jr}} ~\text{satisfy constraints}~ \eqref{constr_nonlinear}; \\
				& \check{H}_{out}^{S_j} - M_j,
				&& ~\text{if there exists}~ r_0 \in \{1,\dots,N^P_j\},\\
				&&& \check{F}_{jr} ~\text{and}~ \check{T}_{ave}^{P_{jr}} ~\text{do not satisfy constraints}~ \eqref{constr_nonlinear},
			\end{aligned}\right.~ \ \ j=1,\dots,N^S-1
	\end{align*}
	is feasible to (HOPnr1). Here for $j=1,\dots,N^S-1,$ $M_j$ is defined to be
\begin{equation*}
			M_j = \check{H}^{S_j}_{out} - \sum_{t=1}^{r_1} \left[f\left(\check{T}_{ave}^{P_{jr}}, Q_{jt}, D_{jt}\right)L_{jt} + \Delta Z_{jt}\right] - \max\{\check{H}_{out}^{P_{jr_1}}, \tilde{H}_{out}^{P_{jr_1}}\} \end{equation*}
if $\check{H}^{S_j}_{out} - \sum_{t=1}^{r_1} \left[f\left(\check{T}_{ave}^{P_{jr}}, Q_{jt}, D_{jt}\right)L_{jt} + \Delta Z_{jt}\right] - U_{jr_1} > 0$, and $0$ otherwise. $r_1$ and $U_{jr}$ are defined by
	\begin{align*}
		r_1 &= \mathop{\arg\max}_{r=1,\dots,N^P_j}\left\{\check{H}^{S_j}_{out} - \sum_{t=1}^r \left[f\left(\check{T}_{ave}^{P_{jr}}, Q_{jt}, D_{jt}\right)L_{jt} + \Delta Z_{jt}\right] - U_{jr}\right\},\\
		U_{jr} &= \left\{
			\begin{aligned}
				& \overline{H}_{out}^{P_{jr}}, && r=1,\dots,N^P_j-1;\\
				& \overline{H}_{in}^{S_{j+1}}, && r=N^P_j.
			\end{aligned}\right.
	\end{align*}
	$\tilde{H}_{out}^{P_{jr_1}}$ is defined by $\tilde{s}$ which gives $\tilde{T}_{out}^{S_j} = \overline{T}_{out}^{S_j}$.

Now, to show the above $\hat{s}$ is feasible to (HOPnr1), suppose $\hat{s}$ defines a solution $\hat{\Psi}$ of (HOPnr1). Note that the only difference between (HOPnr1) and (HOPnr2) are nonlinear constraints on the friction. For $j=1,\dots,N^S-1$, if for all $r=1,\dots,N^P_j$, $\check{F}_{jr}$ and $\check{T}_{ave}^{P_{jr}}$ satisfy the constraints \eqref{constr_nonlinear}, then we have $\hat{F}_{jr} = \check{F}_{jr}$, $r=1,\dots,N^P_j$. Therefore the feasibility of $\hat{\Psi}$ with this $j$ is obvious due to the feasibility of $\check{s}$. For the case that there exists some $r_0$ such that $\check{F}_{jr}$ and $\check{T}_{ave}^{P_{jr}}$ do not satisfy the constraints \eqref{constr_nonlinear}, we only need to prove that $\hat{H}_{out}^{S_j}$, $\hat{H}_{in}^{S_{j+1}}$ and $\hat{H}_{out}^{P_{jr}}$ satisfy the inequalities \eqref{constr_inoutH}, \eqref{constr_boundHin}, \eqref{constr_boundHout} and \eqref{constr_transitionHead}, $r=1,\dots,N^P_j-1$.

	If $M_j = 0$, then we have $\hat{H}_{out}^{S_j} = \check{H}_{out}^{S_j}$. Thus the constraints \eqref{constr_boundHout} are satisfied. According to the constraints \eqref{constr_newnonlinear}, we also conclude that for each $r=1,\dots,N^P_j$,
	\begin{align*}
		\hat{H}_{out}^{P_{jr}} &= \check{H}_{out}^{S_j} - \sum_{t=1}^r \left[f\left(\check{T}_{ave}^{P_{jr}}, Q_{jt}, D_{jt}\right)L_{jt} + \Delta Z_{jt}\right] \\
		& \geq \check{H}_{out}^{S_j} - \sum_{t=1}^r \left[\check{F}_{jr} + \Delta Z_{jt}\right] = \check{H}_{out}^{P_{jr}}, \\
		\hat{H}_{out}^{P_{jr}} &= \check{H}_{out}^{S_j} - \sum_{t=1}^r \left[f\left(\check{T}_{ave}^{P_{jr}}, Q_{jt}, D_{jt}\right)L_{jt} + \Delta Z_{jt}\right] \\
		&\leq \check{H}_{out}^{S_j} - \sum_{t=1}^{r_1} \left[f\left(\check{T}_{ave}^{P_{jr}}, Q_{jt}, D_{jt}\right)L_{jt} + \Delta Z_{jt}\right]-U_{jr_1}+U_{jr} \leq U_{jr}.
	\end{align*}
	Therefore the bound constraints \eqref{constr_boundHin} and \eqref{constr_transitionHead} are also satisfied.

	For the case that $\check{H}^{S_j}_{out} - \sum_{t=1}^{r_1} \left[f\left(\check{T}_{ave}^{P_{jr}}, Q_{jt}, D_{jt}\right)L_{jt} + \Delta Z_{jt}\right] - U_{jr_1} > 0$, we have $M_j > 0$ due to the feasibility of $\check{s}$ and $\tilde{s}$ on \eqref{constr_transitionHead}. It follows that $\hat{H}_{out}^{S_j} \leq \check{H}_{out}^{S_j}$. Based on Lemma \ref{lem_monoT} and the monotonicity of $f$, we have
	\begin{align*}
		\hat{H}_{out}^{P_{jr}} &= \check{H}_{out}^{S_j} - M_j - \sum_{t=1}^{r} \left[f\left(\check{T}_{ave}^{P_{jr}}, Q_{jt}, D_{jt}\right)L_{jt} + \Delta Z_{jt}\right] \\
		&\leq \check{H}_{out}^{S_j} - M_j - \sum_{t=1}^{r_1} \left[f\left(\check{T}_{ave}^{P_{jr}}, Q_{jt}, D_{jt}\right)L_{jt} + \Delta Z_{jt}\right] - U_{jr_1} + U_{jr} \\
		&= U_{jr} - U_{jr_1} + \max\{\check{H}_{out}^{P_{jr_1}}, \tilde{H}_{out}^{P_{jr_1}}\} \leq U_{jr},~ \ \ r=1,\dots,N^P_j,\\
		\hat{H}_{out}^{P_{jr}} &= \sum_{t=r+1}^{r_1} \left[f\left(\check{T}_{ave}^{P_{jr}}, Q_{jt}, D_{jt}\right)L_{jt} + \Delta Z_{jt}\right] +\max\{\check{H}_{out}^{P_{jr_1}}, \tilde{H}_{out}^{P_{jr_1}}\}\\
		&\geq \tilde{H}_{out}^{S_j} - \sum_{t=1}^{r} \left[f\left(\tilde{T}_{ave}^{P_{jr}}, Q_{jt}, D_{jt}\right)L_{jt} + \Delta Z_{jt}\right] = \tilde{H}_{out}^{P_{jr}},~ \ \ r =0,\dots, r_1-1, \\
		\hat{H}_{out}^{P_{jr}} &= -\sum_{t=r_1+1}^{r} \left[f\left(\check{T}_{ave}^{P_{jr}}, Q_{jt}, D_{jt}\right)L_{jt} + \Delta Z_{jt}\right] +\max\{\check{H}_{out}^{P_{jr_1}}, \tilde{H}_{out}^{P_{jr_1}}\}\\
		&\geq \check{H}_{out}^{S_j} - \sum_{t=1}^{r} \left[\check{F}_{jr} + \Delta Z_{jt}\right] = \check{H}_{out}^{P_{jr}},~ \ \ r =r_1+1,\dots, N^P_j.
	\end{align*}
	With the feasibility of $\check{s}$ and $\tilde{s}$, the constraints \eqref{constr_boundHin}, \eqref{constr_boundHout} and \eqref{constr_transitionHead} are satisfied. Finally, for all the cases analyzed in this proof, we have \begin{equation*}
		\check{H}_{in}^{S_{j+1}} \geq \hat{H}_{in}^{S_{j+1}},~ \ \ \hat{H}_{out}^{S_{j}} \leq \check{H}_{out}^{S_{j}}.
	\end{equation*}
	The constraints \eqref{constr_inoutH} are satisfied for all $j=1,\dots,N^S-1$ since
	\begin{equation*}
		\hat{H}_{in}^{S_j} + \hat{x}_j H^{CP}_j + \Delta \hat{H}^{SP}_j \geq \check{H}_{in}^{S_j} + \check{x}_j H^{CP}_j + \Delta \check{H}^{SP}_j \geq \check{H}_{out}^{S_j} \geq \hat{H}_{out}^{S_j}.
	\end{equation*}
	Therefore $\hat{s}$ is feasible to (HOPnr1) and
	\begin{equation*}
		C(\check{x}, \Delta \check{H}^{SP},\ \ \Delta \check{T}) = C(\hat{x}, \Delta \hat{H}^{SP}, \Delta \hat{T}).
	\end{equation*}
	This completes the proof.
	\Halmos
\endproof

	\proof{Proof of Theorem \ref{thm_equivalence}}
	Suppose $\check{s}^*$ is an optimal solution of (HOPnr2). Denote $C^*$ as the optimal value of (HOPnr1). One one hand, we have
	\begin{equation*}
		C(\check{x}^*, \Delta \check{H}^{SP*}, \Delta \check{T}^*) \leq C^*,
	\end{equation*} since (HOPnr2) is a relaxation of (HOPnr1). On the other hand, based on Lemma \ref{lem_sameopt}, we can obtain a feasible solution $\hat{s}^*$ of (HOPnr1) from $\check{s}^*$ such that
	\begin{equation*}
		C(\check{x}^*, \Delta \check{H}^{SP*}, \Delta \check{T}^*) \geq C(\hat{x}^*, \Delta \hat{H}^{SP*}, \Delta \hat{T}^*).
	\end{equation*}
	Thus we obtain
	\begin{equation*}
		C^* \leq C(\hat{x}^*, \Delta \hat{H}^{SP*}, \Delta \hat{T}^*) \leq C(\check{x}^*, \Delta \check{H}^{SP*}, \Delta \check{T}^*) \leq C^*,
	\end{equation*}
	which proves the statement.
	\Halmos
\endproof

	\proof{Proof of Theorem \ref{thm_finite}}
	The feasibility of the subproblem at each iteration is guaranteed by the preprocessing procedure. Branching on fractional solutions is implemented to generate two subproblems with disjoint feasible regions. So no two nodes in the branch-and-bound tree of Algorithm \ref{alg_branch-and-bound} have the same bounds. The following set
	\begin{equation*}
		\left\{(\underline{z},\overline{z}) \mid 0 \leq \underline{x}_j \leq \overline{x}_j \leq N^{CP}_j,~ 0 \leq \underline{y}_j \leq \overline{y}_j \leq N^{SP}_j,~ j=1,...,N^S-1 \right\}
	\end{equation*}
	contains finite elements. It means that the branch-and-bound tree has a finite number of nodes. Thus the iteration for searching the tree is also finite, which completes the proof.
	\Halmos
\endproof

	\proof{Proof of Theorem \ref{thm_global}}
	We prove it by contradiction. It is obvious that Algorithm \ref{alg_branch-and-bound} always returns an incumbent solution $\tilde{s}$ and $GUB$ of the problem (HOP)  if it is feasible. Suppose $s^*$ is the global optimal solution of the problem (HOP) and $C(x^*, \Delta H^{SP*}, \Delta T^*) < C(\tilde{x}, \Delta \tilde{H}^{SP}, \Delta \tilde{T})$ holds. Since branching will not affect the feasible region of the problem (HOP), we have that there is a node $(\underline{z},\overline{z},LB)$ where $\underline{x} \leq x^* \leq \overline{x}$ pruned by Algorithm \ref{alg_branch-and-bound}. Note that the node whose feasible region contains $s^*$ can not be pruned by bounding or infeasibility. So we obtain that the $\check{x}$ in the global minimizer $\check{s}$ of (HOPnr1($\underline{x},\overline{x}$)) satisfies the integer constraint. Let $\hat{s}$ be the feasible solution of the problem (HOP) achieved from $\check{s}$ by Proposition \ref{prop_upperbound}. Since $\hat{s}$ satisfies that
	\begin{equation*}
		\hat{z} = \lceil \check{z} \rceil = \check{z},~ \ \ \Delta \hat{H}^{SP} = \Delta \check{H}^{SP},\ \ \Delta \hat{T} = \Delta \check{T},
	\end{equation*}
	it follows from the global optimality of $\check{s}$ that
	\begin{equation*}
		C(\tilde{x}, \Delta \tilde{H}^{SP}, \Delta \tilde{T}) = GUB \leq C(\hat{x}, \Delta \hat{H}^{SP}, \Delta \hat{T}) = C(\check{x}, \Delta \check{H}^{SP}, \Delta \check{T}) \leq
		C(x^*, \Delta H^{SP*}, \Delta T^*),
	\end{equation*}
	which contradicts our assumption. So we complete the proof.
	\Halmos
\endproof

	\proof{Proof of Lemma \ref{lem_u0upper}}
	Let
	\begin{equation*}
		h(u) = -\sum_{t=1}^{r_1-r_0}\left[f(\phi_{jt}^{r_0} u+\psi_{jt}^{r_0}, Q_{j,r_0+t}, D_{j,r_0+t}) L_{j,r_0+t} + \Delta Z_{j,r_0+t}\right].
	\end{equation*}
	Note that $f$ is monotonically decreasing and $L_{j,r_0+r} \geq 0 ,r=1,...,r_1-r_0$, we can easily verify that $h$ is monotonically increasing. Based on the constraints \eqref{constr_hchangepart},\eqref{constr_nonlinear},\eqref{constr_outTpart},\eqref{constr_aveTpart} and the feasibility of $\tilde{u}$, we have that
	\begin{equation*}
		h(T_{out}^{P_{jr_0}}) = H_{out}^{P_{jr_1}} - H_{out}^{P_{jr_0}} \leq \overline{H}_{out}^{P_{jr_1}} - \underline{H}_{out}^{P_{jr_0}} = h(\tilde{u}).
	\end{equation*}
	Therefore we obtain $T_{out}^{P_{jr_0}} \leq \tilde{u}$. According to the constraint \eqref{constr_outTpart}, we can prove the statement of the upper bound of $T_{out}^{S_j}$. This completes the proof.
	\Halmos
\endproof

	\proof{Proof of Lemma \ref{lem_nleqfeasible}}
	Under Assumption \ref{assump_strict}, function $h$ defined in the proof of Lemma \ref{lem_u0upper} is strictly monotonically increasing. Let $\tilde{s}$ be a feasible solution to (HOPnr1($\underline{z}$, $\overline{z}$)). Then we have that
	\begin{equation*}
		h(\tilde{T}_{out}^{P_{jr_0}}) = \tilde{H}_{out}^{P_{jr_1}} - \tilde{H}_{out}^{P_{jr_0}} \leq \overline{H}_{out}^{P_{jr_1}} - \underline{H}_{out}^{P_{jr_0}} < h(\hat{u}) \Rightarrow \tilde{T}_{out}^{P_{jr_0}} < \hat{u}.
	\end{equation*}
	Besides, $h$ is continuous in $[\tilde{T}_{out}^{P_{jr_0}},\hat{u}]$ since $f$ is convex. So there exists a unique $\tilde{u}$ such that
	\begin{equation*}
		\tilde{T}_{out}^{P_{jr_0}} < \tilde{u} < \hat{u},~ \ \ h(\tilde{u}) =  \overline{H}_{out}^{P_{jr_1}} - \underline{H}_{out}^{P_{jr_0}}.
	\end{equation*}
	It follows that $\tilde{u}$ is the unique solution of \eqref{nonlineareqn}. The last inequality in the lemma is further proved by $\tilde{u} < \hat{u}$, which completes the proof.
	\Halmos
\endproof

	\proof{Proof of Lemma \ref{lem_r0>r1}} We prove the statement by contradiction. Suppose $\tilde{s}$ is feasible to (HOPnr1($\underline{z}$, $\overline{z}$)). Let
	\begin{equation*}
	h'(u) = -\sum_{t=1}^{r_0-r_1}\left[f(\phi_{jt}^{r_1} u+\psi_{jt}^{r_1}, Q_{j,r_1+t}, D_{j,r_1+t}) L_{j,r_1+t} +\Delta Z_{j,r_1+t}\right].
	\end{equation*}
	Similar to the proof of Lemma \ref{lem_u0upper}, we know that $h'$ is monotone. Moreover, we know that $\hat{u}$ is an upper bound of the variable $T_{out}^{P_{jr_1}}$. Then we have that
	\begin{equation*}
		h'(\hat{u}) < \underline{H}_{out}^{P_{jr_0}} - \overline{H}_{out}^{P_{jr_1}} \leq \tilde{H}_{out}^{P_{jr_0}} - \tilde{H}_{out}^{P_{jr_1}} = h'(\tilde{T}_{out}^{P_{jr_1}}) \Rightarrow \hat{u} < \tilde{T}_{out}^{P_{jr_1}},
	\end{equation*}
	which contradicts the feasibility of $\tilde{s}$. This completes the proof.
	\Halmos
\endproof

	\proof{Proof of Theorem \ref{thm_preprocess}}
	We first prove that the strengthening step \ref{algstep_tightH} is correct. Then we verify the correctness of steps \ref{algstep_infeasible1}, \ref{algstep_infeasible2} and \ref{algstep_infeasible3}. We finally give the feasible solution $\tilde{s}$ satisfies the condition in the statement.
	
	From Lemmas \ref{lem_u0upper} and \ref{lem_nleqfeasible} and the constraints \eqref{constr_outTpart}, we have that the upper bound strengthening on variable $T_{out}^{S_j}$ in step \ref{algstep_tightT} of Algorithm \ref{alg_preprocess} is correct and $\overline{T}_{out}^{S_j}$ is monotonically decreasing during the preprocessing procedure. We show the correctness of tightening variable $H_{out}^{S_j}$ by contradiction. Suppose there exists a feasible solution $\tilde{s}$ of (HOPnr1($\underline{z}$, $\overline{z}$)) before preprocessing and $\tilde{H}_{out}^{S_j} < \underline{H}_{out}^{S_j}$ holds, where $\underline{H}_{out}^{S_j}$ is obtained after step \ref{algstep_tightH}. According to steps \ref{algstep_fixH} and \ref{algstep_liftH}, it is true that
	\begin{align*}
		\tilde{H}_{out}^{S_j} < \underline{H}^{S_j}_{out} &= \max_{r=0,\dots,N^P_j}\left\{\underline{H}_{out}^{P_{jr}}+\sum_{t=1}^r\left[f(\phi_{jt}^{0} \overline{T}_{out}^{S_j}+\psi_{jt}^{0}, Q_{jt}, D_{jt}) L_{jt} +\Delta Z_{jt}\right]\right\}.
	\end{align*}
	Then there exists some $r_2$ such that
	\begin{align*}
		\tilde{H}_{out}^{S_j} &= \tilde{H}_{out}^{P_{jr_2}}+\sum_{t=1}^{r_2}\left[f(\phi_{jt}^{0} \tilde{T}_{out}^{S_j}+\psi_{jt}^{0}, Q_{jt}, D_{jt}) L_{jt} +\Delta Z_{jt}\right] \\
		&< \underline{H}_{out}^{P_{jr_2}}+\sum_{t=1}^{r_2}\left[f(\phi_{jt}^{0} \overline{T}_{out}^{S_j}+\psi_{jt}^{0}, Q_{jt}, D_{jt}) L_{jt} +\Delta Z_{jt}\right].
	\end{align*}
	From the monotonicity of $f$, it follows that $\tilde{H}_{out}^{P_{jr_2}} < \underline{H}_{out}^{P_{jr_2}}$ or $\tilde{T}_{out}^{S_j} > \overline{T}_{out}^{S_j}$. Either of them contradicts the feasibility of $\tilde{s}$.
	The correctness of the tightening of $H_{in}^{S_j}$ and $T_{in}^{S_j}$ in step \ref{algstep_tightH} is guaranteed by the constraints \eqref{constr_inoutH}, \eqref{constr_inoutT} and \eqref{constr_temprisebound}.
	
	The infeasibility judgements in steps \ref{algstep_infeasible1} and \ref{algstep_infeasible2} are true due to Lemmas \ref{lem_r0>r1} and \ref{lem_nleqfeasible}. It is obvious that step \ref{algstep_infeasible3} is also correct.
	
	When Algorithm \ref{alg_preprocess} terminates, we show that $\tilde{s}$ with
	\begin{align*}
		\tilde{z} = \overline{z},~ \Delta \tilde{H}_j^{SP} = \overline{y}_jH_j^{SP},~  \tilde{H}_{out}^{S_j} = \underline{H}_{out}^{S_j},~ j=1,...,N^S-1,
	\end{align*}
	and $\Delta T_j$ is determined by $T_{out}^{S_j} = \overline{T}_{out}^{S_j}$ and the equality constraints \eqref{constr_outTpart} and \eqref{constr_aveTpart}. Therefore $\tilde{s}$ satisfies the condition in the theorem. Note that after each iteration $j$, we obtain a partial solution with $H_{out}^{S_j} = \underline{H}_{out}^{S_j},~ T_{out}^{S_j} = \overline{T}_{out}^{S_j}$. The partial solution is equal to the corresponding components of $\tilde{\Psi}$ defined by $\tilde{s}$. Therefore the feasibility of $\tilde{s}$ can be guaranteed by the procedure of Algorithm \ref{alg_preprocess}.
	This completes the proof.
	\Halmos
\endproof

\end{APPENDIX}

\end{document}